\documentclass[12pt]{amsart}
\usepackage[margin=1in]{geometry}
\usepackage{latexsym}

\usepackage[utf8]{inputenc}
\usepackage{amsfonts}
\usepackage{amsmath}
\usepackage{amssymb}
\usepackage{amsthm}
\usepackage{commath}
\usepackage{subfig}
\usepackage{float}
\usepackage{algorithm}
\usepackage{algorithmic}
\usepackage[hyphens]{url}
\usepackage[hyphenbreaks]{breakurl}
\usepackage[colorlinks]{hyperref}
\hypersetup{
	linkcolor=NavyBlue,
	citecolor=NavyBlue,
	filecolor=Red,
	urlcolor=NavyBlue,
	menucolor=Red,
	runcolor=Red}
\usepackage{thmtools}
\usepackage{thm-restate}
\usepackage{cleveref}
\usepackage[hyphens]{url}
\usepackage[hyphenbreaks]{breakurl}

\usepackage[dvipsnames]{xcolor}

\usepackage{todonotes}

\usepackage[style=numeric]{biblatex} 
\usepackage{amsmath}

\theoremstyle{definition}
\newtheorem{counter}{ctr}[section]
\newtheorem{theorem}[counter]{Theorem}
\newtheorem{definition}[counter]{Definition}
\newtheorem{conjecture}[counter]{Conjecture}
\newtheorem{corollary}[counter]{Corollary}
\newtheorem{proposition}[counter]{Proposition}
\newtheorem{lemma}[counter]{Lemma}

\newtheorem{question}[counter]{Question}

\newcommand{\N}{\mathbb{N}}

\newcommand{\R}{\mathbb{R}}

\newcommand{\E}{\mathbb{E}}
\newcommand{\Prob}{\mathbb{P}}

\newcommand{\Var}{\text{var}}

\begin{document}

\title{Towards the Gaussianity of Random Zeckendorf Games}

\author[Cheigh, Dantas e Moura, Jeong, Lehmann Duke, Milgrim, Miller, Ngamlamai]{Justin Cheigh, Guilherme Zeus Dantas e Moura, Ryan Jeong, Jacob Lehmann Duke, Wyatt Milgrim, Steven J. Miller, Prakod Ngamlamai}

\maketitle

\begin{abstract}
    Zeckendorf proved that any positive integer has a unique decomposition as a sum of non-consecutive Fibonacci numbers, indexed by $F_1 = 1, F_2 = 2, F_{n+1} = F_n + F_{n-1}$. Motivated by this result, Baird, Epstein, Flint, and Miller \cite{baird2018zeckendorf} defined the two-player Zeckendorf game, where two players take turns acting on a multiset of Fibonacci numbers that always sums to $N$. The game terminates when no possible moves remain, and the final player to perform a move wins. Notably, \cite{baird2018zeckendorf} studied the setting of random games: the game proceeds by choosing an available move uniformly at random, and they conjecture that as the input $N \to \infty$, the distribution of random game lengths converges to a Gaussian.
    
    We prove that certain sums of move counts is constant, and find a lower bound on the number of shortest games on input $N$ involving the Catalan numbers. The works \cite{baird2018zeckendorf} and Cuzensa et al. \cite{cusenza2020bounds} determined how to achieve a shortest and longest possible Zeckendorf game on a given input $N$, respectively: we establish that for any input $N$, the range of possible game lengths constitutes an interval of natural numbers: every game length between the shortest and longest game lengths can be achieved. 
    
    We further the study of probabilistic aspects of random Zeckendorf games. We study two probability measures on the space of all Zeckendorf games on input $N$: the uniform measure, and the measure induced by choosing moves uniformly at random at any given position. Under both measures that in the limit $N \to \infty$, both players win with probability $1/2$. We also find natural partitions of the collection of all Zeckendorf games of a fixed input $N$, on which we observe weak convergence to a Gaussian in the limit $N \to \infty$. We conclude the work with many open problems. 
\end{abstract}

\tableofcontents

\section{Introduction}

The Fibonacci numbers are widely considered to be the most interesting and well-known recursive sequence in mathematics. In this article, we shall index the Fibonacci numbers by $F_1 = 1$, $F_2 = 2$, and for general $n \geq 3$, $F_n = F_{n-1} + F_{n-2}$. Zeckendorf proved the following fundamental theorem: the decomposition in the proceeding theorem is referred to as the \textbf{Zeckendorf decomposition} of the positive integer $N$.
\begin{theorem}[\cite{zeckendorf1972representations}]
Every positive integer $N$ can be decomposed uniquely into a sum of distinct, non-consecutive Fibonacci numbers.
\end{theorem}

Inspired by this result, the authors of \cite{baird2018zeckendorf} constructed the \textbf{two-player Zeckendorf game}.
\begin{definition}[\cite{baird2018zeckendorf}]
Given input $N \in \mathbb N$, the Zeckendorf game is played on a multiset of Fibonacci numbers, initialized at $\mathcal S = \{F_1^N\}$. On each turn, a player can act on the multiset by performing one of the following moves if it is available.
\begin{enumerate}
    \item If we have two consecutive Fibonacci numbers $F_{k-1}, F_k$ for some $k \geq 2$, then we can replace them by $F_{k+1}$, denoted $F_{k-1} \wedge F_k \to F_{k+1}$.
    \item If we have two instances of the same Fibonacci number $F_k$, then
    \begin{enumerate}
        \item If $k=1$, we can play $F_1 \wedge F_1 \to F_2$.
        \item If $k=2$, we can play $F_2 \wedge F_2 \to F_1 \wedge F_3$.
        \item If $k \geq 3$, we can play $F_k \wedge F_k \to F_{k-2} \wedge F_{k+1}$.
    \end{enumerate}
\end{enumerate}
The two players alternate turns until no playable moves remain. The last player to move wins the game. 
\end{definition}
Observe that the moves of the game are consistent with the Fibonacci recurrence: we either \textit{combine} two consecutive terms, or \textit{split} terms with multiple instances. Perhaps more intuitively, we can understand the game as acting on a row of bins, with bin $k$ corresponding to the Fibonacci number $F_k$ and its height being the multiplicity of $F_k$ in the multiset.

\subsection{Prior Work}
The article \cite{baird2018zeckendorf} introduces the two-player Zeckendorf game, determine upper and lower bounds on the length of a game on input $N$ (showing in particular that the game always terminates), and shows non-constructively that Player 2 has the winning strategy for all $N \geq 2$. In particular, they provide the following explicit formula for the length of the shortest Zeckendorf game on input $N$, achieved by only playing combine moves.
\begin{theorem}[\cite{baird2018zeckendorf}] \label{shortest_game}
The number of combine moves in any Zeckendorf game on input $N$ is is $N-Z(N)$. Furthermore, any such shortest game terminates in $N - Z(N)$ moves, where $Z(N)$ is the number of terms in the Zeckendorf decomposition of $N$.
\end{theorem}
The works \cite{li2020deterministic} and \cite{cusenza2020bounds} successively improved the upper bound of \cite{baird2018zeckendorf} on the length of a Zeckendorf game on fixed input $N$; the former article finds a deterministic game which has longest possible length for input $N$, while the latter generalizes this paradigm. We frequently make use of the following two results from \cite{cusenza2020bounds} in our arguments. We note that although the following result provides a strategy to achieve a longest game, finding a convenient closed form for the length of the longest game for non-Fibonacci input $N$ remains open (with the case of Fibonacci input treated by \cite{cusenza2020bounds}).
\begin{theorem}[\cite{cusenza2020bounds}] \label{longest_game}
The longest game on any $N$ is achieved by applying split moves or combine $1$s (in any order) whenever possible, and, if there is no split or combine $1$ move available, combine consecutive indices from smallest to largest.
\end{theorem}

\begin{theorem}[\cite{cusenza2020bounds}] \label{type_a_only}
A Zeckendorf game on input $N$ can be played with strictly splitting and combine $1$ moves if and only if $N = F_k - 1$ for some $k \geq 2$.
\end{theorem}
Finally, we remark that analogous two-player games have been developed for other recurrences: the work \cite{GeneralizedZeckendorfGame} extends \cite{baird2018zeckendorf} by defining and studying such games for recursive sequences defined by linear recurrence relations of form $G_n = \sum_{i=1}^k c G_{n-i}$ ($c = k-1=1$ yielding the Fibonacci numbers), again giving lower and upper bounds on game lengths (including showing termination) and showing non-constructively that Player 2 has a winning strategy, while \cite{boldyriew2020extending} similarly studies recurrences of form $a_{n+1} = na_n + a_{n-1}$.

\subsection{Notation and Conventions}
We let $C_1$ denote the combine move $F_1 \wedge F_1 \to F_2$, and (for $k \geq 2$) let $C_k$ denote the combine move $F_{k-1} \wedge F_k \to F_{k+1}$. Let $S_2$ denote the splitting move $2F_2 \to F_1 \wedge F_3$, and (for $k \geq 3$) let $S_k$ denote the splitting move $2F_k \to F_{k-2} \wedge F_{k+1}$. We prefix a particular type of move with $M$ to denote the number of such moves (e.g. $MC_1$ denotes the number of $C_1$'s in a game). 

We let $Z(N)$ denote the number of terms in the Zeckendorf decomposition of $N$. We loosely refer to the number of instances of $F_k$ as the \textbf{height} of bin $k$, denoted $h_k$; it will usually be clear from context at which point in the game the quantity $h_k$ refers to. When discussing the height of bin $k$ after a specific number $m$ of moves in the game, we notate this by $h_k(m)$. For $\lambda \in \mathbb N$, we shall also occasionally use the shorthand $[\lambda] = \{1, 2, \dots, \lambda\}$.

In this work, we shall generally work under the assumption that $F_n \leq N < F_{n+1}$ for some $n \in \mathbb N$ (i.e., $n$ is the index of the largest Fibonacci number that is no larger than $N$)\footnote{This is why we have elected to deviate from notation traditionally used in papers concerning this game.}. While proving Theorem \ref{all_lengths}, we occasionally refer to moves $C_1, S_2, \dots, S_{n-1}$ as \textbf{Type A moves}, and all other moves (namely, moves $C_k$ for $k \geq 2$) as \textbf{Type B moves}. The work \cite{cusenza2020bounds} also achieved an understanding of precisely when playing strictly Type A moves throughout the whole game is possible.

Finally, the present article furthers the study of random Zeckendorf games. Here, we let $\Omega_N$ denotes the (finite) collection of all Zeckendorf games on input $N$, with $\mathcal F_N = 2^{\Omega_N}$ the associated $\sigma$-algebra, and express a given Zeckendorf game $\mathcal G \in \Omega_N$ as a (finite) sequence of $\lambda$ moves, written as $\mathcal G = (M_1, M_2, \dots, M_\lambda)$. We study two probability measures to complete the space $(\Omega_N, \mathcal F_N)$: the uniform measure $\mu_N$, defined by 
\begin{align*}
    \mu_N(\mathcal G) = \frac{1}{|\Omega_N|} \text{ for all } \mathcal G \in \Omega_N
\end{align*}
and the probability measure $\Prob_N$ induced by choosing, at every configuration along a given game, uniformly at random among available moves, defined by
\begin{align*}
    \Prob_N(\mathcal G) = \prod_{k=1}^\lambda \frac{1}{\text{num. playable moves after } (M_1,\dots,M_{k-1})} \text{ for } \mathcal G = (M_1, M_2, \dots, M_\lambda) \in \Omega_N.
\end{align*}
All of the results we derive in this context apply to both probability spaces.

\subsection{Main Results} 

The work \cite{cusenza2020bounds} determines an upper bound on the length of a game on input $N$. We improve this upper bound using similar techniques as in the work \cite{baird2018zeckendorf}.
\begin{restatable}{theorem}{upperBound}
\label{upperboundours}
The length of a Zeckendorf game on input $N$ is upper-bounded by
    \[\left\lfloor\varphi^2N-Z_I(N)-2Z(N)+(\varphi-1)\right\rfloor\]
    where $Z_I(N)$ represents the index sum of the Zeckendorf decomposition of $N$. Furthermore, the bound is sharp for infinitely many $N$.
\end{restatable}
Much of our work was inspired by the following conjecture (the only one still unresolved in the paper it was introduced in), initially posed by \cite{baird2018zeckendorf}, which concerns distributional properties of the length of random Zeckendorf games on input $N$ in the limit $N \to \infty$. 

\begin{conjecture}[\cite{baird2018zeckendorf, ,li2020deterministic}] \label{gaussianity}
In the limit $N \to \infty$, the distribution of the number of moves in a random Zeckendorf game on input $N$ converges to a Gaussian, with expectation and variance approximately $0.215N$.\footnote{The authors of \cite{baird2018zeckendorf} posed this conjecture based on numerical data gathered from 9,999 simulations of a random game with $n=18$. The authors of \cite{li2020deterministic} gathered further numerical evidence with a sample of 1,000 games with $n=1,000,000$. We ran a brute force enumeration over all possible games for $n\leq 18$ and found the distribution of lengths appeared to be Gaussian. This is a slightly different problem than the random game though is closely related.}
\end{conjecture}
As such, many of our main results have largely arisen from attempting to understand those aspects of Zeckendorf games which may potentially aid in resolving the aforementioned conjecture (and in striving to determine what such aspects are). First, we have the following lower bound on the number of shortest Zeckendorf games of length $N$. Intuitively, if the distribution of random game lengths were indeed Gaussian, this should be an extreme underestimate compared to the number of ways to achieve other game lengths (shortest games involve the fewest number of decisions, so one might naturally expect that the probability of achieving one via a random game is larger than longer games), yet it still explodes in $N$.

\begin{restatable}{theorem}{catBound}
\label{lower_bound}
Let $F_n \leq N < F_{n+1}$. Then the number of shortest Zeckendorf games with input $N$ is at least $\prod_{k=1}^{n-2} \text{Cat}(F_k)$, where $\text{Cat}(F_k)$ denotes the $F_k\textsuperscript{th}$ Catalan number.
\end{restatable}
The shortest game and longest game were studied in \cite{baird2018zeckendorf} and \cite{cusenza2020bounds}, respectively. It is natural to ask whether every game length between the shortest and longest game length is achievable: we resolve this in the affirmative.
\begin{restatable}{theorem}{allLengths}
\label{all_lengths}
For any input $N$ to the Zeckendorf game, let $M$ denote the length of the longest Zeckendorf game with input $N$. Then for any $m$ satisfying $N - Z(N) \leq m \leq M$, there exists a Zeckendorf game of length $m$ on input $N$. In other words, the set of achievable game lengths constitutes an interval in the natural numbers.
\end{restatable}
We also study the winning odds of players in the limit $N \to \infty$ of infinite input when studying random Zeckendorf games, for which one might expect that both players win with probability $1/2$ in the limit if if Conjecture \ref{gaussianity} holds as the variance of the conjectured Gaussian grows with $N$. We establish that this is indeed true by proving a much more general result: we can understand Theorem \ref{many_players} as saying that in the limit of infinite input, a $Z$-player random Zeckendorf game is fair, in the sense that all $Z$ players have the same probability of winning.
\begin{restatable}{theorem}{manyPlayers}
\label{many_players}
For any integer $Z \geq 1$ and $z \in \{0,1,\dots,M-1\}$, we have that
\begin{align*}
    \lim_{N \to \infty} \mu_N(\text{Game length equals } z \text{ mod } Z) = \lim_{N \to \infty} \Prob_N(\text{Game length equals } z \text{ mod } Z) = \frac{1}{Z}.
\end{align*}
\end{restatable}
\noindent Taking $Z = 2$ in Theorem \ref{many_players} above yields the following result for the classical two-player Zeckendorf game.
\begin{restatable}{theorem}{halfHalf}
\label{half_half}
For the two-player Zeckendorf game, in the limit $N \to \infty$ under both probability measures $\mu_N$ and $\Prob_N$, Player 1 and Player 2 are equally likely to win. Explicitly,
\begin{align*}
    & \lim_{N \to \infty} \mu_N(\text{Player 1 wins}) = \lim_{N \to \infty} \mu_N(\text{Player 2 wins}) = \frac{1}{2}, \\
    & \lim_{N \to \infty} \Prob_N(\text{Player 1 wins}) = \lim_{N \to \infty} \Prob_N(\text{Player 2 wins}) = \frac{1}{2}.
\end{align*}
\end{restatable}

Finally, we establish that there exist natural ways to partition the collection of Zeckendorf games $\Omega_N$ on input $N$ so that the distribution of game lengths over the corresponding classes are nearly Gaussian with high probability in the limit $N \to \infty$. The construction of the subsets $\mathcal R_N^{\mathcal P} \subset \Omega_N$ and $\mathcal R_N^{\mathcal S} \subset \Omega_N$, and the sets $\mathcal A_N(\mathcal R)$, is elaborated in Propositions \ref{prefix_partition} and \ref{suffix_partition}.

\begin{restatable}{theorem}{weakMixtures}
\label{weak_mixtures}
For $\mathcal R \in \mathcal R_N^{\mathcal P}$, let $F_N^{\mathcal R}(x): \R \to [0,1]$ denote the distribution function corresponding to game lengths in $\mathcal A_N(\mathcal R)$ over the conditional distribution induced by $\Prob_N$, normalized to have expectation $0$ and variance $1$. Let $\Phi: \R \to [0,1]$ denote the distribution function of the standard normal. Then for any $\epsilon > 0$,
\begin{align*}
    \lim_{N \to \infty} \Prob_N\left( \sup_{x \in \R} \left| F_N^{\mathcal R}(x) - \Phi(x) \right| \geq \epsilon \right) = 0.
\end{align*}
Similarly, for $\mathcal R \in \mathcal R_N^{\mathcal S}$, let $F_N^{\mathcal R}(x): \R \to [0,1]$ denote the distribution function corresponding to game lengths in $\mathcal A_N(\mathcal R)$ over the conditional distribution induced by $\Prob_N$, normalized to have expectation $0$ and variance $1$. Then for any $\epsilon > 0$,
\begin{align*}
    \lim_{N \to \infty} \Prob_N\left( \sup_{x \in \R} \left| F_N^{\mathcal R}(x) - \Phi(x) \right| \geq \epsilon \right) = 0.
\end{align*}
The analogous results hold for the uniform measure $\mu_N$.
\end{restatable}

\section{Structural Results}

In this section, we include some straightforward, but fundamental results concerning the nature of the Zeckendorf game; some of these will be invoked in proofs of deeper theorems.

\subsection{Combinatorial Observations} We begin by exploring some basic properties of the Zeckendorf game, observable by studying deterministic subroutines of moves. The following simple result mirrors techniques in \cite{baird2018zeckendorf}

\begin{proposition} \label{all_states_reachable}
Consider any decomposition of $N$ into a sum of (possibly non-distinct, non-consecutive) Fibonacci numbers: this decomposition can be achieved via a sequence of combine moves from the starting configuration of the Zeckendorf game.
\end{proposition}

\begin{proof}
We ``play the game in reverse": consider the configuration corresponding to this decomposition, and construct a sequence of moves by always taking the game piece not in the first bin (i.e., $F_1$) and farthest out, and replace it as the result of a combine move. Specifically, if $k \geq 3$, then replace $F_k$ by $\{F_{k-2}, F_{k-1}\}$; if $k = 2$, replace with $2F_1$. Then reverse all the moves to get a Zeckendorf game from the initial state to this state.
\end{proof}

The following is an easy consequence of Theorem \ref{longest_game} above, which states that the longest game paradigm from the starting position extends to intermediate game positions on input $N$ which are given by converting $F_n-1$ instances of $F_1=1$ into the Zeckendorf decomposition of $F_n-1$, where $n$ denotes the index of the largest Fibonacci number in the Zeckendorf decomposition of the input $N$.
\begin{lemma} \label{longest_game_2}
Let $n$ denote the index of the largest Fibonacci number in the Zeckendorf decomposition of the input $N$. A longest Zeckendorf game from an intermediate configuration given by converting $F_n-1$ instances of $1$ into the Zeckendorf decomposition of $F_n-1$ is given by greedily playing any Type A move whenever possible, and if no such Type A move can be played, play the available Type B move with the smallest index.
\end{lemma}
\begin{proof}
If a game achieved by playing Type A moves whenever possible from this configuration were not maximal (i.e., there existed a Zeckendorf game of strictly larger length), then by initially playing the longest game on input $F_n-1$ via all Type A moves (possible by Theorem \ref{type_a_only}), we can play the game exactly according to Theorem \ref{longest_game} but fail to achieve a game of maximal length, contradicting Theorem \ref{longest_game}.
\end{proof}

Using similar techniques as in \cite{baird2018zeckendorf}, we derive the following results in order to improve the upper bound found in \cite{cusenza2020bounds}.
\begin{lemma}\label{movesum}
Let $n$ be the largest summand in the Zeckendorf decomposition of $N$, we get that for any $2\leq k \leq n-1$, the following sum is constant:
		\[MS_k + MC_k + MC_{k+1} + \cdots + MC_{n-1}.\]
\end{lemma}
\begin{proof}
    Consider the following relabeling of the board:
	\begin{center} \small
		\begin{tabular}{ |c|c|c|c|c|c|c|c| } 
			\hline
			$F_1$ & $\cdots$ & $F_k$ & $F_{k+1}-1$ & $F_{k+2}-2$ & $F_{k+3}-4$ & $F_{k+4}-7$ & $\cdots$\\ 
			\hline
		\end{tabular}
	\end{center}
	where after the $k\textsuperscript{th}$ bin, the value of a bin is equal to one less than the sum of the values of the two bins which precede it. We get that only the moves $S_k, C_k, C_{k+1}, \dots$ can change the weighted sum of the tokens by the relabeled values, and each of these moves reduce the sum by 1. Since we have a fixed initial sum $N$ and a fixed ending sum depending on the Zeckendorf Decomposition of $N$, we get that the sum of those moves must be constant regardless of how the game is played.
\end{proof}
\begin{lemma}\label{exceptiondiff}
    For any Zeckendorf game starting with $N$ tokens, \[MC_1 - MS_2 \approx (2-\varphi)N\] with approximation error $\leq \varphi-1$.
\end{lemma}
\begin{proof}
    Similarly, we prove this with a relabeling of the board
	\begin{center}
		\begin{tabular}{ |c|c|c|c|c|c| } 
			\hline
			$2$ & $3$ & $5$ & $\cdots$ & $F_{k+1}$ & $\cdots$ \\ 
			\hline
		\end{tabular}
	\end{center}
	and observing that the sum of token values goes from $2N$ to $\sim \varphi N$ with the sum decreasing by 1 only by performing $C_1$ and increasing by 1 only by performing $S_2$. Note that the final sum is equal to shifting each Zeckendorf summands of $N$ forward by one, which is approximated by multiplying each by $\varphi$. By Binet's formula, we have $F_{k} = \frac{\varphi^{k+1}-(-1/\varphi)^{k+1}}{\sqrt{5}}$ due to how we index the Fibonacci sequence. Thus, the error of approximating the summand $F_{k+1}$ with $\varphi F_k$ is \[ \left\lvert\frac{(\varphi^{k+1}-(-1/\varphi)^{k+1})\varphi - (\varphi^{k+2} - (-1/\varphi)^{k+2})}{\sqrt{5}}\right\rvert  =  \frac{\varphi^2+1}{\varphi^{k+2}\sqrt{5}} \]
    The largest error happens when $N = F_1 + F_3 + F_5 + \cdots$ with error at most \[\frac{\varphi^2+1}{\varphi^3\sqrt{5}}\left(\sum_{i=0}^{\infty}\frac{1}{\varphi^{2i}}\right) = \frac{\varphi^2+1}{\varphi^3(1-1/\varphi^2)\sqrt{5}} = \frac{1}{\varphi} = \varphi - 1\]
    which yields the desired.
\end{proof}
\noindent As a corollary, we prove Theorem \ref{upperboundours}.
\upperBound*
\begin{proof}
    Using the relabeling of the board given by
    \begin{center}
		\begin{tabular}{ |c|c|c|c|c|c| } 
			\hline
			$3$ & $4$ & $5$ & $\cdots$ & $k+2$ & $\cdots$ \\ 
			\hline
		\end{tabular}
    \end{center}
    we get that \begin{align*}
        2MC_1+2MC_2+3MC_3+4MC_4+&\cdots  = 3N - 2Z(N) - Z_I(N) \\
        + MS_3 + MS_4 + &\cdots
    \end{align*}
    Applying Lemma \ref{exceptiondiff}, we get \begin{align*}
        MC_1+2MC_2+3MC_3+4MC_4+&\cdots  \leq 3N - 2Z(N) - Z_I(N) \\
        +(2-\varphi)N - (\varphi-1) + MS_2 + MS_3 + MS_4 + &\cdots
    \end{align*}
    Thus, if we subtract the excess $MC_2 + 2MC_3 + 3MC_4 + \cdots$ from the left hand side, we get the upper bound \begin{align*}
        MC_1+MC_2+MC_3+MC_4+&\cdots  \leq (1+\varphi)N - 2Z(N) - Z_I(N)+(\varphi-1) \\
        + MS_2 + MS_3 + MS_4 + &\cdots
    \end{align*} 
    which we round down as the number of all moves performed is an integer. Note that if we consider $N = F_n - 1$ then we get that there are no $C_2,C_3,$ performed within the longest game. Furthermore, if we consider $N = F_{2n}-1$, then our approximation error is less than $1$ and thus must be sharp after rounding down.
\end{proof}

Recall that \cite{baird2018zeckendorf} proved non-constructively that Player $2$ always has a winning strategy for any $N \geq 3$: finding such a winning strategy remains open. It is generally believed (\cite{miller2022convo}) that the key to such a strategy lies in understanding ``parity swaps": distinct sequences of moves of differing length which yield the same effect on the board. The following definition follows from the easy observation that whenever playable, the sequences of moves
\begin{align*}
    & S_k \to S_{k-1} \to \cdots \to S_{k-\ell+1} \to C_{k-\ell}
    & C_{k-\ell} \to S_{k-\ell+1} \to \cdots \to S_{k-1} \to S_k
\end{align*}
both have the same effect on the board as the move $C_k$, for some $k \geq 2$.
\begin{definition} \label{prefix_suffix}
For any $\ell \geq 0$ and $k \geq 2$, call a sequence of moves of form $S_k \to S_{k-1} \to \cdots \to S_{k-\ell+1} \to C_{k-\ell}$ an $(\ell,k)$-\textit{prefix}, and a sequence of moves of form $C_{k-\ell} \to S_{k-\ell+1} \to S_{k-\ell+2} \to \cdots \to S_k$ an $(\ell,k)$-\textit{suffix}. We call an $(\ell,k)$-prefix a \textit{prefix of} $C_k$, and an $(\ell,k)$-suffix a \textit{suffix of} $C_k.$
\end{definition}
It should be emphasized that an $(\ell,k)$-prefix or $(\ell,k)$-suffix corresponds to an equivalent action as the combine move $C_k$, for all lengths $\ell \geq 0$ and $k \geq 2$; the combine move is thus ``expanded" via a sequence of $\ell$ splitting moves with contiguous indices. The next result captures the intuition that the variation in game lengths is entirely due to the parity swaps described in Definition \ref{prefix_suffix}, namely by describing arbitrary Zeckendorf games via \textit{permutations of suffixes.}
\begin{proposition} \label{suffix}
Any Zeckendorf game on input $N$ can be achieved by taking a shortest game, expanding combine moves via suffixes, then shifting the splitting moves.
\end{proposition}
\begin{proof}
For an arbitrary Zeckendorf game, greedily take the earliest split move of a game, move it back to be played as early as possible, and compress it into a combine move. Proceed similarly until we only have combine moves, from which we achieve the original game by reversing the compressions and playing split moves later.
\end{proof}
We can interpret the statement of Proposition \ref{suffix} as saying that we can greedily embed an arbitrary Zeckendorf game on input $N$ into a shortest game on input $N$ in a natural way.

\subsection{Shortest Games} Proposition \ref{suffix} suggests that a study of shortest Zeckendorf games might be fruitful, as any particular Zeckendorf game can be understood as an extension and permutation of a particular shortest game. We first observe the following.

\begin{proposition} \label{all_combines}
Shortest games are exactly those games which strictly use combine moves. Furthermore, such a game exists for any input $N$, and the multiset of combine moves for any such shortest game is unique.
\end{proposition}
\begin{proof}
A move decreases the number of pieces by at most one, so $N-Z(N)$ lower-bounds the number of moves necessary, achieved exactly by those games using strictly combine moves; such games exist by Proposition \ref{all_states_reachable}. To establish uniqueness of the multiset of combine moves for any such game, say $F_n \leq N < F_{n+1}$, and study bin $k$ for $1 \leq k \leq n$. Moves affecting $h_k$ are known precisely: letting the Zeckendorf decomposition of $N$ be denoted $(z_1, z_2, \dots, z_n)$ (where $z_i \in \{0,1\}$), this yields the system
\begin{align*}
    N-2MC_1 - MC_2 & = z_1 \\
    MC_1 - MC_2 - MC_3 & = z_2 \\
    & \cdots \\
    MC_{n-3} - MC_{n-2} - MC_{n-1} & = z_{n-2} \\
    MC_{n-2} - MC_{n-1} & = z_{n-1} \\
    MC_{n-1} & = z_n = 1
\end{align*}
from which it easily follows that this system must have a unique solution.
\end{proof}

\begin{proposition} \label{prop:small_combines}
For $n = n(N)$ and any $\delta \in (0,1)$,
\begin{align*}
    \lim_{N \to \infty} \max_{\mathcal G \in \Omega_N} \frac{\#\{C_k \in \mathcal G, k > \delta n\}}{\#\{\text{Combine moves in } \mathcal G\}} = \lim_{N \to \infty} \max_{\mathcal G \in \Omega_N} \frac{\#\{C_k \in \mathcal G, k > \delta n\}}{N-1} = 0.
\end{align*}
\end{proposition}
\begin{proof}
For $\mathcal G \in \Omega_N$, consider all combine moves $C_k$ with $k > \delta n$; each such combine move corresponds to a particular token ``jumped," for which we consider all combine moves with index $\leq \delta n$ that led this token to land in this position, of which there must be at least $\lfloor \delta n / 2 \rfloor$. The sets of combine moves with index $\leq \delta n$ corresponding to distinct combine moves with index $> \delta n$ are observed to be disjoint, so
\begin{align*}
    \frac{\#\{C_k \in \mathcal G, k > \delta n\}}{N-1} \leq \frac{1}{\lfloor \delta n / 2 \rfloor+1}
\end{align*}
which vanishes as $N \to \infty$.
\end{proof}

Proposition \ref{all_combines} yields the following interesting lower bound\footnote{We suspect the lower bound of Theorem \ref{lower_bound} to be somewhat loose, as much is lost when crudely pursuing the interweaving of the Dyck paths $\pi_2, \pi_3, \dots, \pi_{n-1}$ (see the proof for details).} on the number of shortest Zeckendorf games with input $N$.
\catBound*
\begin{proof}
It suffices to study $N = F_n$, since the number of distinct shortest Zeckendorf games is increasing in $N$: for input $N \neq F_n$, one can first play a shortest Zeckendorf game on input $F_n$, then proceed by always playing the rightmost available combine move to achieve a shortest game on input $N$ by Proposition \ref{all_combines}. (See Lemma \ref{add_1}, where we understand this as incrementing instances of $F_1$s after playing a shortest game on $F_n$.) The Zeckendorf decomposition is $(z_1, z_2, \dots, z_{n-1}, z_n) = (0, 0, \dots, 0, 1)$, where $z_i$ denotes the number of instances of $F_i$ in the decomposition. By Proposition \ref{all_combines}, there exists a unique multiset of combine moves constituting a shortest game: by solving the system above, this multiset is defined by
\begin{align*}
    (MC_1, MC_2, \dots, MC_{n-1}) = (F_{n-2}, F_{n-3}, F_{n-4}, \dots, F_3, F_2, F_1, F_0) 
\end{align*}
where we let $F_0 = 1$. A permutation of these moves constitutes a game if and only if every move is valid, i.e., no move would force the height of any bin to become negative. Specifically, at any intermediate point in the sequence, the number of $C_1$'s played is no less than the sum of the number of $C_2$'s and $C_3$'s played (bin $2$ is nonnegative), the number of $C_2$'s is no less than the number of $C_3$'s and $C_4$'s played (bin $3$ is nonnegative), and so on, to the number or $C_{n-2}$'s being no less than the number of $C_{n-1}$'s played (bin $n-1$ is nonnegative). (We need not study bin $1$ or bin $n$, which will necessarily always have nonnegative height.)

Let us restrict our attention to the moves $C_1, C_2, C_3$ (moves affecting the height of bin $2$): the number of permutations of these moves such that the number of $C_1$'s played being no less than the sum of the number of $C_2$'s and $C_3$'s performed holds at any point in the game is in bijective correspondence with (up-down) Dyck paths on $j=F_{n-2}$ ($C_1 \leftrightarrow U$; $C_2, C_3 \leftrightarrow D$), the number of which is $\text{Cat}(F_{n-2})$. Similarly, by studying moves affecting bin $k \geq 2$, we achieve bijective correspondences with Dyck paths on $j = F_{n-k}$. 

For any choice of Dyck paths $\pi_2, \pi_3, \dots, \pi_{n-1}$ on $j = F_{n-2}, F_{n-3}, \dots, F_1$, respectively, there exists a shortest Zeckendorf game on input $N$ where the ordering of the relevant moves is consistent with the bijections described above. To construct such a game, begin by placing $2F_{n-2}$ moves along a line, labeling $F_{n-2}$ of them as $C_1$ in a manner consistent with $\pi_2$ ($C_1 \leftrightarrow U$). Among the $F_{n-2} = F_{n-3} + F_{n-4}$ unlabeled moves, label $F_{n-3}$ of them as $C_2$ (importantly, including the first $D$ move) in a manner consistent with $\pi_3$ (not too many unlabeled moves between consecutive instances of $C_2$) and the other $F_{n-4}$ as $C_3$. Now add $F_{n-5}$ instances of $C_4$ along this line to complete $\pi_3$ (include all missing $C_4 \leftrightarrow D$ moves) while respecting $\pi_4$ (not too many unlabeled moves between consecutive instances of $C_3$). Specifically, construct a labeling of the $F_{n-3} = F_{n-4} + F_{n-5}$ $D$ moves in $\pi_3$ with $F_{n-4}$ $C_3$'s and $F_{n-5}$ $C_4$'s such that the first $D$ is labeled $C_3$, and there are not too many unlabeled $D$ moves between consecutive instances of $C_3$ (with respect to $\pi_4$). Insert $C_4$'s to be adjacent to established instances of $C_2$ and $C_3$ to be consistent with this labeling. 

Continue by similarly adding, for $k \geq 5$, $F_{n-k-1}$ instances of $C_k$, completing $\pi_{k-1}$ while respecting $\pi_k$, until we add $F_0 = 1$ instance of $C_{n-1}$ such that we complete $\pi_{n-1}$. This results in a shortest Zeckendorf game with the ordering of the relevant moves being consistent with $\pi_2, \pi_3, \dots, \pi_{n-1}$.
\end{proof}

\section{The Set of Possible Game Lengths Constitute An Interval}

In this section, we prove Theorem \ref{all_lengths}, which we restate below.

\allLengths*

We begin by establishing some intermediate results that we shall invoke in the proof of the main theorem. In the first lemma, in discussing the position given by the Zeckendorf decomposition of $N-1$, we refer to the terminal position of the Zeckendorf game when played on input $N-1$.
\begin{lemma} \label{add_1}
Consider the Zeckendorf game on input $N$, satisfying $F_n \leq N < F_{n+1}$, from the position given by the Zeckendorf decomposition on $N-1$ (as specified above) with an additional instance of $1$. There is a unique sequence of moves from this configuration to the Zeckendorf decomposition of $N$, all of which are combine moves. Furthermore, the number of such combine moves performed is bounded by $\lfloor n/2 \rfloor$.
\end{lemma}

\begin{proof}
Since no moves can be played from the Zeckendorf decomposition on $N-1$, any playable move from this position (on input $N$) is necessarily either $C_1$ or $C_2$, possible if and only if $F_1 = 1$ or $F_2 = 2$ is in the Zeckendorf decomposition of $N-1$, respectively (such cases are disjoint, since the Zeckendorf decomposition on $N-1$ does not contain consecutive Fibonacci numbers); otherwise, no moves can be played. We study both cases.
\begin{itemize}
    \item If the Zeckendorf decomposition of $N-1$ contains an instance of $F_1$, then after playing $C_1$, it is easy to see that the only possible move is $C_3$ (iff the decomposition of $N-1$ contains an $F_3$), then $C_5$ (iff the decomposition of $N-1$ contains an $F_5$), and so on, until we exhaust all playable moves.
    \item If the Zeckendorf decomposition of $N-1$ contains an instance of $F_2$, then after playing $C_2$, it is easy to see that the only possible move is $C_4$ (iff the decomposition of $N-1$ contains an $F_4$), then $C_6$ (iff the decomposition of $N-1$ contains an $F_6$), and so on, until we exhaust all playable moves.
\end{itemize}
In both cases, it is straightforward to confirm that we cannot play strictly more than $\lfloor n/2 \rfloor$ such combine moves, as otherwise there must be an instance of $F_k$ for $k$ either $n+1$ or $n+2$ after completing this sequence of moves, a contradiction on $N < F_{n+1}$.
\end{proof}

We shall also frequently use the following lemma. Intuitively, this states that if we have isolated the Zeckendorf game to a suffix of bins all of height $0$ or $1$, and this suffix is separated from earlier bins by a bin of height $0$ (say bin $k$), then we can ignore bins $1, \dots, k$ for the remainder of the game.
\begin{lemma} \label{no_back}
For some $1 \leq k \leq n$, let $(x_1, x_2, \dots, x_k, x_{k+1}, \dots, x_n)$ denote the heights of bins $1, \dots, n$ at some point during a Zeckendorf game, with $x_i$ denoting the height of bin $i$. Assume $x_i \in \{0,1\}$ for $k \leq i \leq n$, $x_k = 0$, no playable moves involving bins $1, \dots, k$ exist, and we play according to Theorem \ref{longest_game}. Then heights $x_1, \dots, x_k$ remain fixed for the rest of the game.
\end{lemma}

\begin{proof}
We prove this on $k = n-j$ by induction on $0 \leq j \leq n-2$. The statement trivially holds if $j \in \{0,1,2\}$; assuming it for all values less than $j \geq 3$, if $x_{k+1} = 0$ or $x_{k+2} = 0$ we can apply the induction hypothesis to $j-1$ or $j-2$, respectively, so assume $x_{k+1} = x_{k+2} = 1$. Take $r \leq n$ to be largest possible such that $x_{k+1} = \cdots = x_r = 1$: the game proceeds by playing according to Theorem \ref{longest_game}, i.e., by playing the sequence of moves 
\begin{align*}
    (C_{k+1} \to S_{k+3} \to S_{k+4} \to \cdots \to S_r) \to (C_{k+1} \to S_{k+3} \to \cdots \to S_{r-2}) \to \cdots
\end{align*}
where the final subsequence of moves, either $C_{k+1}$ or $C_{k+1} \to S_{k+3}$, depends on the parity of $r-k$. It follows immediately by studying the moves involved that $x_1, \dots, x_k$ remain unchanged during this sequence. Following this sequence, we have $x_r = 0$: invoke the induction hypothesis on $k=r$ (i.e., $n-r < j$) afterwards to complete the proof.
\end{proof}

We now proceed with the proof of Theorem \ref{all_lengths}.
\begin{proof}[Proof of Theorem \ref{all_lengths}]
We have confirmed this statement for $N \leq F_6-1 = 12$ via a computer check. Thus, for $N \geq F_6 = 13$, assume the statement holds for all input sizes at most $N-1$: we aim to show the result holds for $N$. We explicitly specify that $F_n \leq N < F_{n+1}$.

Let $I_1'$ denote the interval (by induction hypothesis) of possible Zeckendorf game lengths for input $N-1$: by Lemma \ref{add_1}, if we include an additional instance of $F_1$ to the Zeckendorf decomposition of $N-1$, there is a sequence of combine moves from the resulting configuration to the Zeckendorf decomposition of $N$. On input $N$, consider initially playing the Zeckendorf game (to completion) as if the input were $N-1$, then executing this sequence of moves to terminate the game. Let $I_1 = [L_1, R_1]$ denote the interval $I_1'$ shifted by the length of this sequence: by the preceding description, it follows that every game length in $I_1$ is achievable for input $N$. Furthermore, $L_1 = N - Z(N)$, since the game length $L_1$ as studied above results from playing strictly combine moves (see Lemma \ref{add_1}), which necessarily yields a shortest game by Proposition \ref{all_combines}, and for which the length is $N - Z(N)$ by Theorem \ref{shortest_game}.

By Theorem \ref{type_a_only}, it is possible to play the Zeckendorf game on input $F_n-1$ strictly using Type A moves. Let $I_2'$ denote the interval (by induction hypothesis) of possible Zeckendorf game lengths for input $F_n-1 \leq N-1$: on input $N$, consider initially playing the Zeckendorf game (to completion) as if the input were $F_n-1$, then executing the longest possible sequence of moves from the resulting position to terminate the game. Let $I_2 = [L_2, R_2]$ denote the interval $I_2'$ shifted by the number of moves of this longest sequence. By Theorem \ref{longest_game}, $R_2$ is necessarily the length of the longest Zeckendorf game on input $N$, since the above approach is consistent with playing Type A moves whenever possible. 

Thus, it suffices to show that $L_2 \leq R_1$ to complete the induction and establish the theorem, as this yields that $I_1 \cup I_2 = [L_1, R_2]$ is an interval of achievable game lengths on input $N$, with the endpoints being the shortest and longest possible game lengths for input $N$.

\medskip

The known game with length $R_1$ corresponds to playing the longest game on input $N-1$, then performing the unique sequence of combine moves to achieve the Zeckendorf decomposition of $N$; by Theorems \ref{longest_game} and \ref{type_a_only}, we can take the first phase of this game (longest game on input $N-1$) as playing the longest game on input $F_n-1$, then playing the longest remaining game (on input $N-1$). The known game with length $L_2$ corresponds to playing the shortest game on input $F_n-1$, then playing the longest remaining game until we achieve the Zeckendorf decomposition of $N$. By Lemma \ref{longest_game_2}, we can take the second phase of this game (playing the longest remaining game) as playing consistent to Theorem \ref{longest_game} until achieving the Zeckendorf decomposition. 

We henceforth denote these games with length $R_1$ and $L_2$ by $\mathcal G_1$ and $\mathcal G_2$, respectively, which we depict as follows; ``longest on $k$" indicates that during this phase, we think of the game as being played on input $k$, and leaving the appropriate number of instances of $F_1 = 1$ in the first bin fixed.
\begin{align*}
    & \mathcal G_1: \left[(\text{Longest on } F_n-1) \to (\text{longest remaining on } N-1)\right] \to (\text{combine moves on } N) \\
    & \mathcal G_2: (\text{Shortest on } F_n-1) \to (\text{longest remaining on } N)
\end{align*}
Denoting the difference between the lengths of the longest and shortest games on input $F_n-1$ by $\ell(n)$, game $\mathcal G_1$ took exactly $\ell(n)$ more moves than game $\mathcal G_2$ on input $F_n-1$. However, $\mathcal G_2$ may take longer afterwards to finish the game on input $N$: we aim to show that the discrepancy in the game lengths after initially playing on input $F_n-1$ is dominated by $\ell(n)$, from which we conclude that $\mathcal G_2$ is no longer than $\mathcal G_1$, i.e., that $L_2 \leq R_1$. In particular, it certainly suffices to show that the first two segments of $\mathcal G_1$ involve at least as many moves as the first two segments of $\mathcal G_2$ to establish the result: this is how we shall proceed. 

By Theorem \ref{longest_game} and Lemma \ref{longest_game_2}, we can choose how we would like to play a longest game after playing the game on input $F_n-1$ (specifically, we can fix an ordering on Type A moves which determines what we play when given multiple Type A moves): until the two games diverge\footnote{We shall assume this does happen, as otherwise the lengths of the second segments of $\mathcal G_1$ and $\mathcal G_2$ are equal, and thus the inequality $L_2 \leq R_1$ is immediate.}, pursue a longest remaining game by always playing the rightmost Type A move whenever a Type A move is playable. Following this ordering on Type A moves, we study the first move on which games $\mathcal G_1$ and $\mathcal G_2$ deviate. This move is necessarily either $C_1$ or $C_2$ in game $\mathcal G_2$ (the move must not have been playable in game $\mathcal G_1$, and thus must involve bin $1$), and by the ordering on Type A moves established before, bins $2, 3, \dots, n$ are either $0$ or $1$ when it is played. We perform casework on which move the two games deviate on.

\medskip

\textbf{Case 1: Move is $C_1$.} In game $\mathcal G_1$, this configuration can be represented by the vector $(1,x_2,x_3,\dots,x_n)$, where $x_i \in \{0,1\}$ for $i \geq 2$ denotes the height of the $i$th bin (the first entry would be a $2$ for game $\mathcal G_2$). Let us first study the setting $x_2 = 0$, and consider what happens after game $\mathcal G_2$ plays $C_1$. If $x_3=0$, then by Lemma \ref{no_back} applied to $k=3$, both games are consistent on bins $4, \dots, n$ so that game $\mathcal G_2$ takes one more move than $\mathcal G_1$ to finish. Otherwise (i.e., $x_3=1$), Lemma \ref{no_back} on $k=2$ yields that game $\mathcal G_1$ works strictly over bins $3, \dots, n$, and Lemma \ref{no_back} on $k=1$ yields that game $\mathcal G_2$ works strictly over the bins $2, \dots, n$ (i.e., bin $1$ becomes irrelevant). Thus, the resulting setting corresponds exactly to Case 2 over the $n-1$ bins $2, \dots, n$; here, we have an upper bound of $n-1$ for the number of additional moves $\mathcal G_2$ takes, for a total upper bound of $(n-1)+1=n$ (with the $C_1$ in $\mathcal G_2$) for the number of additional moves $\mathcal G_2$ takes.

Thus, assume $x_2 = 1$, and let position $k+1$, with $k \geq 2$, denote the first index that is $0$. Study the length-$k$ prefix $(1,1,1,\dots,1)$ (in game $\mathcal G_1$; first entry is $2$ in game $\mathcal G_2$) with no zero entries; we can exactly describe how the two games proceed when playing according to the longest game paradigm of Theorem \ref{longest_game}.
\begin{itemize}
    \item Game $\mathcal G_1$: We play the sequence of moves
    \begin{align*}
        (C_2 \to S_3 \to S_4 \to \cdots \to S_k) \to (C_2 \to S_3 \to \cdots \to S_{k-2}) \to \cdots
    \end{align*}
    continuing similarly until one of the first two bins is empty (so we cannot play $C_2$). In general, shorten the contiguous $1$s down two indices and repeat until one of the first two bins is empty (the bin which ends up empty depends on the parity of $k$).
    
    \item Game $\mathcal G_2$: We play the sequence of moves 
    \begin{align*}
        (C_1) \to (S_2) \to (S_3 \to C_1) \to (S_4 \to S_2) \to \cdots
    \end{align*}
    continuing similarly until we play all splitting moves $S_j$ for $j \leq k$. In general, when first playing $S_j$, we play the sequence $S_j \to S_{j-2} \to \cdots$ (final move is $S_2$ or $C_1$, depending on parity of $j$), and this continues until position $k$.
\end{itemize}
After these subroutines, it is straightforward to confirm that bin $k$ is empty, bins $k+1, \dots, n-1$ are all $0$ or $1$ with heights agreeing between $\mathcal G_1$ and $\mathcal G_2$, and no playable moves involving bins $1, \dots, k$ exist. Thus, both games proceed strictly over bins $k+1$ to $n-1$ by Lemma \ref{no_back}, and perform the same sequence of moves; we need only compute the difference in the lengths of these two subroutines on this prefix of length $k$.
\begin{itemize}
    \item The length of the sequence of moves in the game of length $L_2$ is
    \begin{align*}
        \begin{cases}
            2 \sum_{j=1}^{(k-1)/2} j + \frac{k+1}{2} & k \text{ odd,} \\
            2 \sum_{j=1}^{k/2} j & k \text{ even.}
        \end{cases}
    \end{align*}
    \item The length of the sequence of moves in the game of length $R_1$ is
    \begin{align*}
        \begin{cases}
            \sum_{j=1}^{(k-1)/2} (k-2j+1) = 2 \sum_{j=1}^{(k-1)/2} j & k \text{ odd,} \\
            \sum_{j=1}^{k/2} (k-2j+1) = \sum_{j=1}^{k/2} (2j-1) = 2 (\sum_{j=1}^{k/2} j) - \frac{k}{2} & k \text{ even.}
        \end{cases}
    \end{align*}
\end{itemize}
We can thus study this difference exactly: the difference is given by 
\begin{align*}
    \begin{cases}
        \left[ 2 \sum_{j=1}^{(k-1)/2} j + \frac{k+1}{2} \right] - 2 \sum_{j=1}^{(k-1)/2} j = \frac{k+1}{2} = \lceil k/2 \rceil & k \text{ odd} \\
        2 \sum_{j=1}^{k/2} j - \sum_{j=1}^{k/2} (2j-1) = \frac{k}{2} & k \text{ even}
        \end{cases}
\end{align*}
so in general, the difference is bounded by $\lceil n/2 \rceil \leq n$.

\medskip

\textbf{Case 2: Move is $C_2$.} In game $\mathcal G_1$, this configuration can be represented by the vector $(0,x_2,x_3,\dots,x_n)$, where $x_i \in \{0,1\}$ for $i \geq 2$ denotes the height of the $i$th bin (the first entry would be a $1$ for game $\mathcal G_2$) and $x_2 = 1$ (since otherwise $C_2$ cannot be played in game $\mathcal G_2$). Let us first study the setting $x_3 = 0$, and consider what happens after game $\mathcal G_2$ plays $C_2$. Applying Lemma \ref{no_back} on $k=3$ for $\mathcal G_1$ and $k=2$ for $\mathcal G_2$ yield that both games work strictly over bins $3, \dots, n$, and this reduces to the same setting on the suffix of bins $3, \dots, n$. Say we reduce the problem to a suffix with length reduced by $2$ in this manner $m$ times, so we study the case where $x_3 = 1$ over $n-2m$ bins: by extracting the bound in the following argument (i.e., the $x_3 = 1$ case) for the number of additional moves $\mathcal G_2$ takes, this yields a bound of $m + (n-2m) \leq n$.

Thus, assume $x_3 = 1$, and let position $k+1$, with $k \geq 3$, denote the first index that is $0$. Study the length-$k$ prefix $(0,1,1,\dots,1)$ (in game $\mathcal G_1$; first entry is $1$ in game $\mathcal G_2$): we can explicitly describe how the two games necessarily proceed when playing according to the longest game paradigm.
\begin{itemize}
    \item Game $\mathcal G_1$: We play the sequence of moves
    \begin{align*}
        (C_3 \to S_4 \to S_5 \to \cdots \to S_k) \to (C_3 \to S_4 \to \cdots \to S_{k-2}) \to \cdots
    \end{align*}
    continuing similarly until either bin $2$ or bin $3$ is empty (so we cannot play $C_3$). In general, shorten the contiguous $1$s down two indices and repeat until either bin $2$ or bin $3$ is empty (the bin which ends up empty depends on the parity of $k$).

    \item Game $\mathcal G_2$: This is exactly the same as game $\mathcal G_1$ in Case 1.
\end{itemize}
After these subroutines, it is straightforward to confirm that bin $k$ is empty, bins $k+1, \dots, n-1$ are all $0$ or $1$ with heights agreeing between $\mathcal G_1$ and $\mathcal G_2$, and no playable moves involving bins $1, \dots, k$ exist. Thus, both games proceed strictly over bins $k+1$ to $n-1$ by Lemma \ref{no_back}, and perform the same sequence of moves; we need only compute the difference in the lengths of these two subroutines on this prefix of length $k$.

The length of the sequence of moves in game $\mathcal G_2$ was computed in Case 1, while the length of the sequence of moves in game $\mathcal G_1$ is
 \begin{align*}
    \begin{cases}
        \sum_{j=1}^{(k-1)/2} (k-2j) = \sum_{j=1}^{(k-1)/2} (2j-1) = 2(\sum_{j=1}^{(k-1)/2} j) - \frac{k-1}{2} & k \text{ odd} \\
        \sum_{j=1}^{(k-2)/2} (k-2j) = \sum_{j=1}^{(k-2)/2} (2j-1) = 2(\sum_{j=1}^{(k-2)/2} j) - \frac{k-2}{2} & k \text{ even.}
    \end{cases}
\end{align*}
so the difference is given by 
\begin{align*}
    \begin{cases}
        2 \sum_{j=1}^{(k-1)/2} j - \left[ 2(\sum_{j=1}^{(k-1)/2} j) - \frac{k-1}{2} \right] = \frac{k-1}{2} = \lfloor k/2 \rfloor & k \text{ odd} \\
        \left[ 2 (\sum_{j=1}^{k/2} j) - \frac{k}{2} \right] - \left[ 2(\sum_{j=1}^{(k-2)/2} j) - \frac{k-2}{2} \right] = k-1 & k \text{ even}
        \end{cases}
\end{align*}
so in general, the difference is bounded by $\max\{\lfloor k/2 \rfloor, k-1\} \leq n-1 \leq n$.

\medskip

We observe that in both cases, the difference in the lengths of the second segments of these games is bounded by $n$. We now show that the difference $\ell(n)$ between the lengths of the longest and shortest games on input $F_n-1$ is at least $n$ for all $n \geq 6$. One can confirm, by playing a longest game according to Theorem \ref{longest_game} on input $N = F_6-1 = 12$, that $\ell(6) = 17 - (12-Z(12)) = 17-9 \geq 6$; it is similarly easy to confirm that $\ell(4), \ell(5) \geq 1$. Since we have $F_{n+1}-1 = (F_n-1)+F_{n-1} \geq (F_n-1)+(F_{n-1}-1)$, one can pursue a game on input $F_{n+1}-1$ by first playing a game on input $F_n-1$, then a game on input $F_{n-1}-1$, and finally performing some fixed sequence of moves to completion. By combining respective shortest games and longest games on input $F_n-1$ and input $F_{n-1}-1$, we observe that $\ell(n+1)-\ell(n) \geq 1$ for any $n \geq 6$, and thus $\ell(n) \geq n$ for all $n \geq 6$ (recall $\ell(6) \geq 6$).

\medskip

Therefore, we have that for all $N \geq F_6 = 13$,
\begin{align*}
    R_1 - L_2 & \geq \ell(n) - \left[(\text{len. longest remaining on $N$}) - (\text{len. longest remaining on $N-1$})\right] \\
    & \geq \ell(n) - n \geq \ell(n) - \ell(n) = 0
\end{align*}
from which we conclude that $L_2 \leq R_1$.
\end{proof}

\section{Winning Odds in the Limit $N \to \infty$}

We dedicate this section to establishing Theorem \ref{half_half}, which follows as an immediate corollary of Theorem \ref{many_players}: recall that Theorem \ref{many_players} is given as follows.
\manyPlayers*
\noindent The case $Z=2$ gives Theorem \ref{half_half}.
\halfHalf*

\subsection{Overview}
We provide a brief sketch of the proof of Theorem \ref{half_half}. We specifically focus on the number of $(1,k)$-prefixes which occur over the course of a game, recalling that whenever applicable, the sequence of moves $(S_k, C_{k-1})$ has the same effect on the game as $C_k$. We partition the collection of games $\Omega_N$ into subsets of games which differ only via such $(1,k)$-prefixes, describe the conditional distribution induced by any such subset (under both $\mu_N$ and $\Prob_N$), and argue that in the limit of infinite input, the sizes of the subsets induced by this partition grow so fast that the probability of Player 1 or Player 2 winning can be reduced to the outcome of a binomial random variable with exploding variance. 

\subsection{Partitioning the Collection of Possible Games}
As observed in Definition \ref{prefix_suffix}, for $k \geq 2$, the sequence\footnote{We deviate from notation earlier in the paper and write move sequences as tuples.} $(S_k, C_{k-1})$ is a $(1,k)$-prefix. We establish some notation to use later: define $\mathcal R_N \in \mathcal F_N$ to be the collection\footnote{We elect to use the notation $\mathcal R_N$ as we think of these games on input $N$ as \textit{representatives} of the corresponding classes $\mathcal A_N(\mathcal R)$ that we define later in this discussion.} of all Zeckendorf games on input $N$ such that (for all $k \geq 2$) any instance of the sequence $(S_k, C_{k-1})$ is immediately preceded by $S_{k+1}$; in other words, $\mathcal R_N$ is the collection of all Zeckendorf games on input $N$ such that there are no $(1,k)$-prefixes for any $k \geq 2$. We can express this collection as
\begin{align*}
    \mathcal R_N = \left\{ (M_1,\dots, M_\lambda) \in \Omega_N : (M_i, M_{i+1}) = (S_k, C_{k-1}) \implies M_{i-1} = S_{k+1} \ \forall \ i \in [\lambda-1], k \geq 2 \right\}.
\end{align*}
For a game $\mathcal R = (M_1, M_2, \dots, M_\lambda) \in \mathcal R_N$, construct the subset of indices $\mathcal I_N(\mathcal R)$ to denote all combine moves in $\mathcal R$, not involved in a sequence $(S_k, C_{k-1})$ for $k \geq 2$, for which the latter bin has height at least $2$ (i.e., combine moves $C_k$ for $k \geq 2$ replaceable with a $(1,k)$-prefix):
\begin{align*}
    \mathcal I_N(\mathcal R) = \left\{i \in [\lambda] : M_i = C_k, \ M_{i-1} \neq S_{k+1}, \ h_k(i) \geq 2 \text{ for some } k \geq 2 \right\}.
\end{align*}
Now, construct the \textit{formal} sequence of moves $\mathcal M(\mathcal R) = (\tilde{M}_1, \tilde{M}_2, \dots, \tilde{M}_\lambda)$ by replacing $M_i$ by a symbol $\mathcal E_k$ (the subscript being the corresponding $k \geq 2$) for all $i \in \mathcal I_N(\mathcal R)$; call $\mathcal M(\mathcal R)$ the \textbf{base sequence} of $\mathcal R \in \mathcal R_N$. Let $\mathcal A_N(\mathcal R) \in \mathcal F_N$ denote the collection of all Zeckendorf games resulting from replacing each instance of $\mathcal E_k$ in $\mathcal M$ by either $C_k$ or the $1$-prefix $(S_k, C_{k-1})$ (for every $k \geq 2$). We establish the following important result, which makes it clear why we have pursued this construction in the manner that we did.

\begin{lemma} \label{partition}
The sets $\{\mathcal A_N(\mathcal R) : \mathcal R \in \mathcal R_N\}$ partition $\Omega_N$.
\end{lemma}

\begin{proof}
We first show that the sets $\mathcal A_N(\mathcal R)$ are disjoint. Take distinct $\mathcal R_1 = (M_1^1,\dots, M_\lambda^1), \mathcal R_2 = (M_1^2,\dots, M_{\lambda'}^2) \in \mathcal R_N$ with base sequences $\mathcal M_1(\mathcal R_1) = (\tilde{M}_1^1, \tilde{M}_2^1, \dots, \tilde{M}_\lambda^1)$ and $\mathcal M_2(\mathcal R_2) = (\tilde{M}_1^2, \tilde{M}_2^2, \dots, \tilde{M}_{\lambda'}^2)$, respectively. Since $\mathcal R_1 \neq \mathcal R_2$, take smallest $i$ for which $M_i^1 \neq M_i^2 \implies \tilde{M}_i^1 \neq \tilde{M}_i^2$, and study the construction of any two games $\mathcal G_1 \in \mathcal A_N(\mathcal R_1)$ and $\mathcal G_2 \in \mathcal A_N(\mathcal R_2)$ as described above. If $\mathcal G_1$ and $\mathcal G_2$ are consistent prior to $\tilde{M}_i^1$ and $\tilde{M}_i^2$ (i.e., for $k \geq 2$, all $\mathcal E_k$ are replaced by the same choice of $C_k$ or $(S_k, C_{k-1})$), which is certainly the only way the two games remain equal up to this point of their construction, then we necessarily produce a difference in the two games on $\tilde{M}_i^1$ and $\tilde{M}_i^2$ if $M_i^1 = \tilde{M}_i^1$ and $M_i^2 = \tilde{M}_i^2$ ($M_i^1 \neq M_i^2$) or $M_i^1 \neq \tilde{M}_i^1$ and $M_i^2 \neq \tilde{M}_i^2$ ($\tilde{M}_i^1 = \mathcal E_{k_1}$, $\tilde{M}_i^2 = \mathcal E_{k_2}$, $k_1 \neq k_2$). If $M_i^1 = \tilde{M}_i^1$ and $M_i^2 \neq \tilde{M}_i^2$ (say $\tilde{M}_i^2 = \mathcal E_k$, so necessarily $M_i^2 = C_k$ and $M_{i-1}^2 \neq S_{k+1}$ for some $k \geq 2$), then game $\mathcal G_2$ is next filled with either $C_k \neq M_i^1$ ($M_i^1 = M_i^2$ if $C_k = M_i^1$) or $(S_k, C_{k-1})$. In the latter case, if the two games were equal after this, then necessarily $(M_{i-1}^1, M_i^1, M_{i+1}^1) = (S_{k+1}, S_k, C_{k-1})$ (by definition of $\mathcal I_N(\mathcal R_1)$ and $\mathcal R_N$, $M_{i+1}^1$ and $M_{i-1}^1$ follow after establishing $M_i^1 = S_k$), contradicting $M_{i-1}^2 \neq S_{k+1}$ (since $M_{i-1}^1 = M_{i-1}^2$ and $M_i^2 = C_k$, $\tilde{M}_i^2 = \mathcal E_k$). 

We therefore conclude that $\mathcal A_N(\mathcal R_1) \cap \mathcal A_N(\mathcal R_2) = \emptyset$, i.e., the sets $\mathcal A_N(\mathcal R)$ for $\mathcal R \in \mathcal R_N$ are disjoint; it remains to show that any game $\mathcal G \in \Omega_N$ is in some set $\mathcal A_N(\mathcal R)$. For $\mathcal G \in \Omega_N$, let $\mathcal R$ be the game resulting from replacing every instance of the sequence $(M, S_k, C_{k-1})$, $M \neq S_{k+1}$ in game $\mathcal G$ by the sequence $(M, C_k)$ (for $k \geq 2$). The resulting game $\mathcal R$ is such that $(M_i, M_{i+1}) = (S_k, C_{k-1}) \implies M_{i-1} = S_{k+1}$ for all $i \in [\lambda]$ and $k \geq 2$: any sequence $(M_{i-1}, M_i, M_{i+1}) = (M, S_k, C_{k-1})$ with $M \neq S_{k+1}$ in $\mathcal R$ necessarily results from having replaced $(S_{k-1}, C_{k-2})$ for $C_{k-1}$ in game $\mathcal G$ (as $(S_k, C_{k-1})$ would have been replaced by $S_{k+1}$ otherwise), but we know this does not occur by the description above, so we indeed have $\mathcal R \in \mathcal R_N$. Also, $\mathcal G \in \mathcal A_N(\mathcal R)$: we can reverse all the replacements $(S_k, C_{k-1}) \leftrightarrow C_k$ made in achieving $\mathcal R$ from $\mathcal G$, since the resulting $C_k$ moves correspond to $\mathcal E_k$ (for some $k \geq 2$) in the base sequence $\mathcal M(\mathcal R)$ as the preceding move is not $S_{k+1}$.
\end{proof}

Therefore, applying Lemma \ref{partition} and the law of total probability,
\begin{align} \label{eq:total_prob_sum}
    & \mu_N(\text{Game length is } z \text{ mod } Z) = \sum_{\mathcal R \in \mathcal 
    R_N} \mu_N\left(\text{Game length is } z \text{ mod } Z \ | \ \mathcal A_N(\mathcal R)\right) \cdot \mu_N\left(\mathcal A_N(\mathcal R)\right) \nonumber \\
    & \Prob_N(\text{Game length is } z \text{ mod } Z) = \sum_{\mathcal R \in \mathcal 
    R_N} \Prob_N\left(\text{Game length is } z \text{ mod } Z \ | \ \mathcal A_N(\mathcal R)\right) \cdot \Prob_N\left(\mathcal A_N(\mathcal R)\right)
\end{align}
so we can reduce proving Theorem \ref{many_players} to establishing that the conditional probabilities for $\mathcal R \in \mathcal R_N$, with respect to both measures, overwhelmingly tend to $1/2$ in the limit.

\subsection{Analysis}
Define the random variable $m_N: \Omega_N \to \mathbb N$ first on $\mathcal R \in \mathcal R_N$ by $m_N(\mathcal R) = |\mathcal I_N(\mathcal R)|$, denoting the number of terms $\mathcal E_k$ (for $k \geq 2$) in the base sequence of $\mathcal R$, then lift to arbitrary $\mathcal G \in \Omega_N$ by letting $m_N(\mathcal G) = m_N(\mathcal R)$ for the unique $\mathcal R \in \mathcal R_N$ such that $\mathcal G \in \mathcal A_N(\mathcal R)$ (see Lemma \ref{partition}).\footnote{In particular, $m_N(\mathcal R) = \log_2(\mathcal A_N(\mathcal R))$.} Fix $\mathcal R \in \mathcal R_N$, and observe that on the event $\mathcal A_N(\mathcal R)$, $m_N(\mathcal G)$ for $\mathcal G \in \mathcal A_N(\mathcal R)$ is the fixed constant $|\mathcal I_N(\mathcal R)|$.  Now construct corresponding Bernoulli random variables $X_1^{\mathcal R}, X_2^{\mathcal R}, \dots, X_{| m_N(\mathcal R)|}^{\mathcal R}: \Omega_N \to \{0,1\}$ for each of the instances of terms of form $\mathcal E_k$ (for $k \geq 2$) in the base sequence of $\mathcal R$: here, $X_i^{\mathcal R}(\mathcal G) = 1$ if and only if $\mathcal G \in \mathcal A_N(\mathcal R)$ and $\mathcal G$ is achieved by the $i$th instance of $\mathcal E_k$ in the base sequence $\mathcal R$ being the $(1,k)$-prefix $(S_k, C_{k-1})$. Say we replaced this $\mathcal E_k$ with $(S_k, C_{k-1})$, and let $n_i$ denote the number of playable moves in the game available in the turn immediately after playing $S_k$. We can say more about the random variables $X_i^{\mathcal R}$ defined here.

\begin{lemma} \label{independence}
Fix some $\mathcal R \in \mathcal R_N$, and define the random variables $X_1^{\mathcal R}, X_2^{\mathcal R}, \dots, X_{|m_N(\mathcal R)|}^{\mathcal R}$ as above. When conditioned on $\mathcal A_N(\mathcal R)$, the random variables $X_1^{\mathcal R}, X_2^{\mathcal R}, \dots, X_{|m_N(\mathcal R)|}^{\mathcal R}$ are independent Bernoulli random variables with variable $X_i^{\mathcal R}$ having parameter $1/2$ under the measure $\mu_N$, and parameter $p_i = \frac{1}{1+n_i}$ under the measure $\Prob_N$. Explicitly,
\begin{align*}
    & \mu_N\left(X_i^\mathcal R = 1 \ | \ \mathcal A_N(\mathcal R) \right) = \frac{1}{2},
    & \Prob_N\left(X_i^\mathcal R = 1 \ | \ \mathcal A_N(\mathcal R) \right) = \frac{1}{1+n_i}.
\end{align*}

\begin{proof}
Fix a random variable $X_i^{\mathcal R}$, and observe that for every particular setting of all other $\mathcal E_k$ terms in $\mathcal R$, there exist exactly two games in $\mathcal A_N(\mathcal R)$ faithful to this setting, corresponding to choosing $C_k$ and the $(1,k)$-prefix $(S_k, C_{k-1})$ for the $i$th such $\mathcal E_k$. It follows immediately that under the uniform measure $\mu_N$, we indeed have $\mu_N\left(X_i^\mathcal R = 1 \ | \ \mathcal A_N(\mathcal R) \right) = \frac{1}{2}$, since in particular there exists a bijection between games in the subset $\mathcal A_N(\mathcal R)$ with $X_i^\mathcal R = 0$ and $X_i^\mathcal R = 1$ and all games in $\Omega_N$ are given equal probability under the measure $\mu_N$.

Under the probability measure $\Prob_N$, it is straightforward to observe that the game replacing the $i$th instance of $\mathcal E_k$ by the $(1,k)$-prefix $(S_k, C_{k-1})$ requires an additional decision with probability $\frac{1}{n_i}$ of yielding the desired $C_{k-1}$, and it thus follows that the parameter $p_i$ of the Bernoulli random variable $X_i^{\mathcal R}$ is given by
\begin{align*}
    \frac{p_i}{1-p_i} = \frac{1}{n_i} \implies p_i = \frac{1}{1+n_i}.
\end{align*}
To establish independence, it suffices to show that for any subset $S \subseteq \left[m_N(\mathcal R)\right]$, we have the identity
\begin{align*}
    \Prob_N\left(X_i^{\mathcal R} = 1 \text{ iff } i \in S \ | \ \mathcal A_N(\mathcal R) \right) & = \prod_{i\in S} \Prob_N\left( X_i^{\mathcal R} = 1 \ | \ \mathcal A_N(\mathcal R) \right) \cdot \prod_{j\notin S} \Prob_N\left( X_j^{\mathcal R} = 0 \ | \ \mathcal A_N(\mathcal R) \right) \\
    & = \prod_{i \in S} p_i \cdot \prod_{j \notin S} (1-p_j).
\end{align*}
We can relate the conditional probabilities $\left\{\Prob_N\left(X_i^{\mathcal R} = 1 \text{ iff } i \in S \ \ \mathcal A_N(\mathcal R) \right): S \in \left[m_N(\mathcal R)\right] \right\}$ whenever $S_1 = S_2 \cup \{j\}$ (for $j \in \left[m_N(\mathcal R)\right]$): all choices for each term $\mathcal E_k$ but one are consistent (namely, $X_j^{\mathcal R} = 1$ for the numerator in the following), and thus
\begin{align*}
    \frac{\Prob_N\left(X_i^{\mathcal R} = 1 \text{ iff } i \in S_1 \ | \ \mathcal A_N(\mathcal R) \right)}{\Prob_N\left(X_i^{\mathcal R} = 1 \text{ iff } i \in S_2 \ | \ \mathcal A_N(\mathcal R) \right)} = \frac{1}{n_j} = \frac{p_j}{1-p_j} = \frac{\prod_{i \in S_1} p_i \cdot \prod_{j \notin S_1} (1-p_j)}{\prod_{i \in S_2} p_i \cdot \prod_{j \notin S_2} (1-p_j)}.
\end{align*}
Now, since we have the identity
\begin{align*}
    \sum_{S \subseteq \left[m_N(\mathcal R)\right]} \Prob_N\left(X_i^{\mathcal R} = 1 \text{ iff } i \in S \ | \ \mathcal A_N(\mathcal R) \right) & = 1 = \prod_{i=1}^{m_N(\mathcal R)} \left( p_i + (1-p_i) \right) \\
    & = \sum_{S \subseteq \left[m_N(\mathcal R)\right]} \left( \prod_{i \in S} p_i \cdot \prod_{j \notin S} (1-p_j) \right)
\end{align*}
and quotients between summands corresponding to two sets differing by one element are the same, the summands on both sides for any subset $S \subseteq \left[m_N(\mathcal R)\right]$ are necessarily equal. More specifically, letting $p_1 = \Prob_N\left(X_i^{\mathcal R} = 1 \text{ iff } i \in \emptyset \ | \ \mathcal A_N(\mathcal R) \right)$ and $p_2 = \prod_{i \in \emptyset} p_i \cdot \prod_{j \notin \emptyset} (1-p_j)$, by incrementally including elements to some $S \subseteq \left[ m_N(\mathcal R) \right]$ we can write the corresponding summands on the left and right hand sides as the same multiple of $p_1$ and $p_2$, respectively. This reduces to $p_1 = p_2$, and thus summands corresponding to the same $S$ are equal. Thus, we have the desired identity for any subset $S \subseteq \left[m_N(\mathcal R)\right]$.
\end{proof}

\end{lemma}
\noindent Lemma \ref{independence} yields the following easy observation.
\begin{corollary} \label{num_bins}
Say $F_n \leq N < F_{n+1}$, fix some $\mathcal R \in \mathcal R_N$, and define the random variables $X_1^{\mathcal R}, X_2^{\mathcal R}, \dots, X_{m_N(\mathcal R)}^{\mathcal R}$ as above. Under the uniform measure $\mu_N$, whenever $\mathcal R \in \mathcal R_N$ is such that $m_N(\mathcal R) > 0$, $\mu_N\left(\text{Player 1 wins} \ | \ \mathcal A_N(\mathcal R) \right) = 1/2$. Under the probability measure $\Prob_N$, when conditioned on the event $\mathcal A_N(\mathcal R)$, there are at most $2n$ distinct values of the parameters $p_i = \frac{1}{1+n_i}$ amongst the random variables $X_i^{\mathcal R}$, and for any such $i$, $\frac{1}{2n-2} \leq p_i \leq 1/2$.
\end{corollary}
\begin{proof}
If $m_N(\mathcal R) > 0$, then the statement $\mu_N\left(\text{Player 1 wins} \ | \ \mathcal A_N(\mathcal R) \right) = 1/2$ follows immediately by applying the law of total probability on all settings of the Bernoulli random variables $(X_2^{\mathcal R}, \dots, X_{m_N(\mathcal R)}^\mathcal R) \in \{0,1\}^{m_N(\mathcal R)-1}$, namely since this further conditioning always yields a conditional probability of $1/2$ (see the proof of Lemma \ref{independence}). At any point of a game $\mathcal G \in \Omega_N$  satisfying $F_n \leq N < F_{n+1}$, there are at most $2n-3$ (and thus certainly at most $2n$) playable moves (and thus at most $2n-3$ distinct values of $n_i$, and thus $p_i$): the combine moves $C_1, C_2, \dots, C_{n-1}$ and the splitting moves $S_2, S_3, \dots, S_{n-1}$. In particular, since $n_i \in [2n-3]$ for all $i \in \left[ m_N(\mathcal R) \right]$, it follows that $\frac{1}{1+(2n-3)} \leq p_i \leq \frac{1}{1+1}$, i.e., $\frac{1}{2n-2} \leq p_i \leq \frac{1}{2}$.
\end{proof}
We will also make use of the following lemma.
\begin{lemma} \label{many_decisions_exp}
Say $F_n \leq N < F_{n+1}$.\footnote{It is perhaps more appropriate to think of $n$ as a function $n(N)$.} For any $c \in (0,\varphi)$ (with $\varphi$ denoting the golden ratio),
\begin{align*}
    & \lim_{N \to \infty} \mu_N \left( m_N(\mathcal G) \geq c^n \right) = 1,
    & \lim_{N \to \infty} \Prob_N \left( m_N(\mathcal G) \geq c^n \right) = 1.
\end{align*}
\end{lemma}

\begin{proof}
Fix $c \in (0, \varphi)$, and consider a game $\mathcal G \in \Omega_N$: for the representative $\mathcal R \in \mathcal R_N$ such that $\mathcal G \in \mathcal A_N(\mathcal R)$, every occurrence of the sequence $(C_1, C_1, C_2)$ corresponds to an element of $\mathcal I_N(\mathcal R)$ (specifically the move $C_2$, as it is not preceded by $S_3$, and the latter two moves are maintained in $\mathcal R$; see the proof of Lemma \ref{partition}), i.e., $m_N(\mathcal G) = |\mathcal I_N(\mathcal R)|$ is at least the number of occurrences of the sequence $(C_1, C_1, C_2)$ in $\mathcal G$. Letting $\mathcal N(\mathcal G)$ denote the number of occurrences of the sequence $(C_1, C_1, C_2)$ in the game $\mathcal G$, it thus suffices to show that
\begin{align*}
    & \lim_{N \to \infty} \mu_N \left( \mathcal N(\mathcal G) \geq c^n \right) = 1,
    & \lim_{N \to \infty} \Prob_N \left( \mathcal N(\mathcal G) \geq c^n \right) = 1.
\end{align*}
Proceed studying the probability measure $\Prob_N$; the analysis carries over exactly\footnote{Certainly, the analysis is much looser than necessary when $\Prob_N$ is replaced with $\mu_N$.} when $\Prob_N$ is replaced with $\mu_N$. By Corollary \ref{num_bins}, at any move where $h_1 \geq 5$, the probability of achieving the sequence $(C_1, C_1, C_2)$ is at least $\left( \frac{1}{2n} \right)^3 = \frac{1}{8n^3}$, and any sequence of three moves can decrease the height of bin $1$ by at most $6$ (via three consecutive $C_1$ moves); observe that the occurrence of $C_2$ in a sequence $(C_1, C_1, C_2)$ within a game $\mathcal G$ necessarily corresponds to an instance of the delimiter $\mathcal E_2$ in the base sequence of $\mathcal G$ (e.g. see the proof of Lemma \ref{partition}). Thus, study the sequences of moves in game $\mathcal G$ (which necessarily exist by the preceding discussion) given by the triples
\begin{align*}
    (M_1, M_2, M_3), (M_4, M_5, M_6), \dots, (M_{3\lfloor N/6 \rfloor-2}, M_{3\lfloor N/6 \rfloor-1}, M_{3\lfloor N/6 \rfloor})
\end{align*}
each of which independently takes on the value $(C_1, C_1, C_2)$ with probability at least $\frac{1}{8n^3}$. Thus, letting $Y_N \stackrel{d}{=} \text{Bin}\left(\lfloor N/6 \rfloor, p = \frac{1}{8n^3} \right)$, we have that
\begin{align*}
    \Prob \left( Y_N \geq c^n \right) \leq \Prob_N \left( \mathcal N(\mathcal G) \geq c^n \right)
\end{align*}
from which $\E[Y_N] = \frac{\lfloor N/6 \rfloor}{8n^3}$ together with a Chernoff bound with $\delta = \frac{1}{2}$ yields for $N$ large,
\begin{align*}
    \Prob \left( Y_N \leq c^n \right) \leq \Prob \left( Y_N \leq \frac{\lfloor N/6 \rfloor}{16n^3} \right) \leq \exp\left( -\frac{\lfloor N/6 \rfloor}{32n^3} \right) \xrightarrow{N \to \infty} 0
\end{align*}
since it is straightforward to verify that $c^n \leq \frac{\lfloor N/6 \rfloor}{16n^3}$ for large $N$ (e.g. use Binet).
\end{proof}

\begin{corollary} \label{many_decisions}
Lemma \ref{many_decisions_exp} immediately yields that for any $d \in \N$,
\begin{align*}
    & \lim_{N \to \infty} \mu_N \left( m_N(\mathcal G) \geq n^d \right) = 1,
    & \lim_{N \to \infty} \Prob_N \left( m_N(\mathcal G) \geq n^d \right) = 1
\end{align*}
i.e., with probability approaching $1$ as $N \to \infty$, the number of decisions made to achieve $\mathcal G$ from the set $\mathcal A_N(\mathcal R)$ containing it is superpolynomial in $n$. This is more convenient for later.
\end{corollary}

\noindent The proof of Lemma \ref{many_decisions_exp} yields the following observations on the likeliest collection $\mathcal A_N(\mathcal R)$ and the number of representative games, which will be referenced again in Section \ref{mixture_gaussians}.
\begin{corollary} \label{cor:likeliest_rep}
The probability of the likeliest collection $\mathcal A_N(\mathcal R)$ vanishes as $N \to \infty$, i.e.,
\begin{align*}
    \lim_{N \to \infty} \max_{\mathcal G \in \Omega_N} \mu_N(\mathcal A_N(\mathcal G)) = \lim_{N \to \infty} \max_{\mathcal G \in \Omega_N} \Prob_N(\mathcal A_N(\mathcal G)) = 0
\end{align*}
and the number of representatives satisfies $|\mathcal R_N| \xrightarrow{N \to \infty} \infty$.
\end{corollary}
\begin{proof}
Again, proceed on the probability measure $\Prob_N$, as the analysis carries over exactly for $\mu_N$. The proof of Lemma \ref{many_decisions} (taking $d=1$) determines there are at least $m \geq n$ instances of $(C_1,C_1,C_2)$ in the first $\lfloor N/6 \rfloor$ moves of a random game $\mathcal G \in \Omega_N$, with probability tending to $1$ in the limit $N \to \infty$: we proceed studying such a game $\mathcal G$. Writing $\mathcal G \in \mathcal A_N(\mathcal R)$, every latter instance of $C_1$ in one of these $m$ sequences is in\footnote{This is in the sense that this instance of $C_1$ is not compressed in achieving $\mathcal R$ from $\mathcal G$ (see the proof of Lemma \ref{partition}), and is not replaced with a term of form $\mathcal E_k$ in achieving the base sequence $\mathcal M(\mathcal R)$.} any game in $\mathcal A_N(\mathcal R)$, and the move $C_2$ could have been played instead of this instance of $C_1$. For $i = 1, \dots, m$, denote $\mathcal B_i^N(\mathcal R) \subseteq \Omega_N$ to be those games consistent with a game in $\mathcal A_N(\mathcal R)$ up to the $i$th such instance of $C_1$, but for which the move $C_2$ is played instead. It is immediately observed that the sets $\mathcal B_i^N(\mathcal R)$ are disjoint and that $\Prob_N(\mathcal A_N(\mathcal R)) \leq \Prob_N(\mathcal B_i^N(\mathcal R))$ for each $i=1,\dots,m$, so we have
\begin{align*}
    n\cdot \Prob_N(\mathcal A_N(\mathcal R)) \leq m\cdot \Prob_N(\mathcal A_N(\mathcal R)) \leq \sum_{i=1}^m \Prob_N(\mathcal B_i^N(\mathcal R)) \leq 1 \implies \Prob_N(\mathcal A_N(\mathcal R)) \leq \frac{1}{n} \xrightarrow{N \to \infty} 0
\end{align*}
so we have $\lim_{N \to \infty} \max_{\mathcal G \in \Omega_N} \Prob_N(\mathcal A_N(\mathcal G)) = 0$, and $|\mathcal R_N| \xrightarrow{N \to \infty} \infty$ follows immediately.
\end{proof}

Finally, we need the following simple result concerning the behavior of a binomial random variable with sufficiently large variance. Certainly, the case $Z=2$ in Lemma \ref{binom} corresponds to studying the expressions $\Prob\left(\mathcal B \text{ is odd}\right)$ and $\Prob\left(\mathcal B \text{ is even}\right)$.
\begin{lemma} \label{binom}
Consider a binomial random variable $\mathcal B = \text{Bin}(m, p)$. For any values of $\epsilon > 0$ and $Z \in \mathbb N$, there exists a constant $N(\epsilon, Z)$ such that if $\Var(\mathcal B) = mp(1-p) \geq N(\epsilon,Z)$ (i.e., if the variance of $\mathcal B$ is sufficiently large), then for any $z \in \{0,1,\dots, Z-1\}$,
\begin{align*}
    \left| \Prob\left(\mathcal B \equiv z \text{ mod } Z \right) - \frac{1}{Z} \right| \leq \epsilon. 
\end{align*}
\end{lemma}

\begin{proof}
The value $\Prob(\mathcal B = k)$ increases on $k \leq \lfloor (m+1)p \rfloor$ and decreases on $k \geq \lfloor (m+1)p \rfloor$, and $\max_{0 \leq k \leq m} \Prob(\mathcal B = k) = \Prob\left(\mathcal B = \lfloor (m+1)p \rfloor \right) \xrightarrow{mp(1-p)\to \infty} 0$ (use Stirling); denoting $p_k = \Prob(\mathcal B = k)$, write $p_0 \leq p_1 \leq \cdots \leq p_{\lfloor (m+1)p \rfloor}$ and $p_{\lfloor (m+1)p \rfloor} \geq p_{\lfloor (m+1)p \rfloor+1}\geq\cdots\geq p_m$. Fix $\epsilon > 0$, and choose $N(\epsilon,Z)$ such that $mp(1-p) \geq N(\epsilon,Z)$ implies $\Prob\left(\mathcal B = \lfloor (m+1)p \rfloor \right) < \frac{\epsilon}{2}$. Certainly, for any distinct values $z_1 < z_2$ in $\{0,1,\dots,Z-1\}$, the number of terms $p_0, p_1, \dots, p_{\lfloor (m+1)p \rfloor}$ with indices equal to $z_1$ modulo $Z$ and equal to $z_2$ modulo $Z$ is either the same or there exists one more such term corresponding to $z_1$: we can switch the ``dominant modulus" by removing the term of largest index equal to either $z_1$ or $z_2$ modulo $Z$, for which the corresponding term $p_k \leq \frac{\epsilon}{2}$ by choice of the constant $N(\epsilon,Z)$. The analogous statement extends to $p_{\lfloor (m+1)p \rfloor}, p_{\lfloor (m+1)p \rfloor+1}, \dots, p_m$, so it follows that we can write
\begin{align*}
    \left| \Prob\left(\mathcal B \equiv z_1 \text{ mod } Z \right) - \Prob\left(\mathcal B \equiv z_2 \text{ mod } Z \right) \right| \leq \epsilon
\end{align*}
which immediately yields the desired statement.
\end{proof}

\noindent We are now ready to proceed with the proof of Theorem \ref{many_players}.
\begin{proof}[Proof of Theorem \ref{many_players}]
Fix an integer $Z \geq 2$ corresponding to the number of players in a $Z$-player Zeckendorf game, and some value $z \in \{0,1,\dots,Z-1\}$. As in the proof of Lemma \ref{many_decisions_exp}, we strictly concern ourselves with the probability measure $\Prob_N$, as the analysis carries over exactly for the uniform measure $\mu_N$ (again, it is much looser than necessary for $\mu_N$). For the probability measure $\Prob_N$, since $m_N(\mathcal G)$ is fixed at $|\mathcal I_N(\mathcal R)| = m_N(\mathcal R)$ on any set $\mathcal A_N(\mathcal R)$ for $\mathcal R \in \mathcal R_N$, for $N \geq N_1(\epsilon)$,
\begin{align} \label{eq:small_decisions_bd}
    \mathop{\sum_{\mathcal R \in \mathcal 
    R_N}}_{m_N(\mathcal R) < 2n^3} \Prob_N\left(\mathcal A_N(\mathcal R)\right) \leq \epsilon.
\end{align}
Consider a fixed representative $\mathcal R \in \mathcal R_N$ for which $m_N(\mathcal R) \geq 2n^3$: by Corollary \ref{num_bins}, there are at most $2n$ distinct values of the parameters $p_i = \frac{1}{1+n_i}$ of the corresponding random variables $X_1^{\mathcal R}, X_2^{\mathcal R}, \dots, X_{m_N(\mathcal R)}^{\mathcal R}$, and thus by the pigeonhole principle, the number of instances of the value of $p_i$ with largest multiplicity is at least $\frac{2n^3}{2n} = n^2 \xrightarrow{N \to \infty} \infty$; henceforth call this $p$. Let us say there are $m \geq n^2$ instances of this value of $p_i$: we can study the sum $\mathcal B_{\mathcal R} = \sum_{i=1}^{m} Y_i$, where the $Y_i$ correspond to those random variables $X_i^{\mathcal R}$ with this corresponding success probability $p$. Furthermore, by Lemma \ref{independence}, the random variables $Y_i$ are independent when conditioned on the event $\mathcal A_N(\mathcal R)$, so $\mathcal B_{\mathcal R} \stackrel{d}{=} \text{Bin}(m, p)$ under this conditional distribution. From the bounds on $p$ from Corollary \ref{num_bins}, it follows that the variance of this binomial random variable, when conditioned on the event $\mathcal A_N(\mathcal R)$, has the exploding lower bound
\begin{align*}
    \Var(\mathcal B_{\mathcal R}) = mp(1-p) \geq \frac{n^2}{2(2n-2)} \xrightarrow{N \to \infty} \infty
\end{align*}
i.e., the binomial random variable $\mathcal B_{\mathcal R}$ has variance exploding in the limit $N \to \infty$ (for any $\mathcal R \in \mathcal R_N$). By Lemma \ref{binom}, the random variable $\mathcal B_{\mathcal R}$ will take a value equal to $z$ modulo $Z$ with probability approaching $\frac{1}{Z}$ as $N \to \infty$. Thus, studying the quantities given by $\Prob_N\left(\text{Game length equals } z \text{ mod } Z \ | \ \mathcal A_N(\mathcal R)\right)$, we can further condition (upon the conditioning $\mathcal A_N(\mathcal R)$) on all random variables $X_i^{\mathcal R}$ not corresponding to those $Y_i$ constituting a summand in the binomial random variable $\mathcal B_{\mathcal R}$ and appeal to the law of total probability:
\begin{align*}
    & \Prob_N\left(\text{Game length equals } z \text{ mod } Z \ | \ \mathcal A_N(\mathcal R)\right) \\
    & = \mathop{\sum_{\text{assignments } A}}_{\text{to } X_i^{\mathcal R} \neq Y_j \forall j} \Prob_N\left(\text{Game length equals } z \text{ mod } Z \ | \ \mathcal A_N(\mathcal R), A\right) \cdot \Prob_N\left(A \ | \ \mathcal A_N(\mathcal R)\right)
\end{align*}
where each term $\Prob_N\left(\text{Game length equals } z \text{ mod } Z \ | \ \mathcal A_N(\mathcal R), A\right)$ is understood as the probability that the length of the game is equal to $z$ mod $Z$ when replacing those $\mathcal E_k$ which correspond to the random variables $Y_i$, while leaving all other moves of the game fixed: this is precisely given by the binomial random variable $\mathcal B_{\mathcal R}$ added to some fixed length determined by the setting of the $X_i^{\mathcal R}$ which are not of the form $Y_j$ for some $j$. Letting $\ell(A)$ denote this length for the assignment $A$, we can thus write
\begin{align*}
    \Prob_N\left(\text{Game length equals } z \text{ mod } Z \ | \ \mathcal A_N(\mathcal R), A\right) = \Prob_N\left(\mathcal B_{\mathcal R} + \ell(A) \text{ is odd} \ | \ \mathcal A_N(\mathcal R), A\right).
\end{align*}
By Lemma \ref{independence}, the random variables $Y_i$ constituting the Bernoulli trials in $\mathcal B_{\mathcal R}$ are independent from the random variables $X_i^{\mathcal R}$ that were fixed when conditioned on $\mathcal A_N(\mathcal R)$ and the assignment $A$, and thus we can apply Lemma \ref{binom} to deduce that these conditional probabilities are arbitrarily close to $\frac{1}{Z}$ for $N \geq N_2(\epsilon,Z)$ by the preceding discussion (for sufficiently large $N_2(\epsilon,Z) \in \mathbb N$). Importantly, this is uniform over all such terms in the sum, in the sense that we can choose $N_2(\epsilon, Z)$ such that we achieve the same guarantee for any such assignment $A$ of binary values to the random variables $X_i^{\mathcal R}$ not constituting $\mathcal B_{\mathcal R}$. Thus, for any $\mathcal R \in \mathcal R_N$ satisfying $m_N(\mathcal R) \geq 2n^3$ and $N \geq N_2(\epsilon, Z)$, we have the bound
\begin{align*}
    \left| \Prob_N\left(\text{Game length equals } z \text{ mod } Z \ | \ \mathcal A_N(\mathcal R)\right) - \frac{1}{Z} \right| \leq \epsilon.
\end{align*}
Therefore, for $N \geq \max\{N_1(\epsilon), N_2(\epsilon,Z)\}$, we can take Equations (\ref{eq:total_prob_sum}) and (\ref{eq:small_decisions_bd}) to achieve the bound on the probability $\Prob_N(\text{Game length equals } z \text{ mod } Z)$ given by
\begin{align*}
    & \Prob_N(\text{Game length is } z \text{ mod } Z) = \sum_{\mathcal R \in \mathcal 
    R_N} \Prob_N\left(\text{Game length is } z \text{ mod } Z  \ | \ \mathcal A_N(\mathcal R)\right) \cdot \Prob_N\left(\mathcal A_N(\mathcal R)\right) \\
    & \leq \epsilon + \mathop{\sum_{\mathcal R \in \mathcal 
    R_N}}_{m_N(\mathcal R) \geq 2n^3} \Prob_N\left(\text{Game length is } z \text{ mod } Z \ | \ \mathcal A_N(\mathcal R)\right) \cdot \Prob_N\left(\mathcal A_N(\mathcal R)\right) \\
    & \leq \epsilon + \mathop{\sum_{\mathcal R \in \mathcal 
    R_N}}_{m_N(\mathcal R) \geq 2n^3} \left( \frac{1}{Z} + \epsilon \right) \cdot \Prob_N\left(\mathcal A_N(\mathcal R)\right) + \epsilon \leq \frac{1}{Z} + 2\epsilon
\end{align*}
and similarly, we have
\begin{align*}
    & \Prob_N(\text{Game length is } z \text{ mod } Z) = \sum_{\mathcal R \in \mathcal 
    R_N} \Prob_N\left(\text{Game length is } z \text{ mod } Z \ | \ \mathcal A_N(\mathcal R)\right) \cdot \Prob_N\left(\mathcal A_N(\mathcal R)\right) \\
    & \geq \mathop{\sum_{\mathcal R \in \mathcal 
    R_N}}_{m_N(\mathcal R) \geq 2n^3} \Prob_N\left(\text{Game length is } z \text{ mod } Z \ | \ \mathcal A_N(\mathcal R)\right) \cdot \Prob_N\left(\mathcal A_N(\mathcal R)\right) \\
    & \geq \mathop{\sum_{\mathcal R \in \mathcal 
    R_N}}_{m_N(\mathcal R) \geq 2n^3} \left( \frac{1}{Z} - \epsilon \right) \cdot \Prob_N\left(\mathcal A_N(\mathcal R)\right) \geq \left( \frac{1}{Z} - \epsilon \right)(1-\epsilon)
\end{align*}
so we conclude that for $N \geq \max\{N_1(\epsilon), N_2(\epsilon,Z)\}$,
\begin{align*}
    \epsilon^2-\frac{Z+1}{Z}\epsilon \leq \Prob_N\left(\text{Game length equals } z \text{ mod } Z \right) - \frac{1}{Z} \leq 2\epsilon
\end{align*}
which yields the desired limit by sending $\epsilon \downarrow 0$.
\end{proof}

\noindent As previously mentioned, the $Z=2$ case of Theorem \ref{many_players} immediately yields Theorem \ref{half_half}. 

\subsection{Extending the Partition on $\Omega_N$} \label{sec:extended_partitions}

Theorem \ref{half_half} establishes that if both players advance randomly, the two-player Zeckendorf game is fair in the limit $N \to \infty$: a key stage of the proof involved partitioning $\Omega_N$ into collections of games where the games in any given collection only differ via certain interchanges of $(1,k)$-prefixes with the corresponding $C_k$. We can naturally extend this partition of $\Omega_N$ (i.e., strictly larger classes in the partition) to encompass arbitrary $(\ell,k)$-prefixes or arbitrary $(\ell,k)$-suffixes. Section \ref{mixture_gaussians} will establish that for these enlarged partitions of $\Omega_N$, we can achieve an analogue of \ref{gaussianity} for the resulting sets in the partition, in the sense that with high probability in the limit, the corresponding distribution is ``nearly Gaussian," in the sense of vanishing Kolmogorov-Smirnov distance (when mean and variance are normalized to be $0$ and $1$, respectively) with the standard normal.\footnote{In this subsection and Section \ref{mixture_gaussians}, we borrow much of the same notation that was used in establishing Theorem \ref{half_half}. It will be clear from context exactly what objects we are referring to.}

We first study the generalization for arbitrary $(\ell,k)$-prefixes. Define $\mathcal R_N^{\mathcal P} \in \mathcal F_N$ to be the collection of all Zeckendorf games on input $N$ such that any combine move $C_k$, for $k \geq 2$, cannot be compressed by a $(1,k)$-prefix (and thus by any prefix of $C_k$). Explicitly,
\begin{align*}
    \mathcal R_N^{\mathcal P} = \left\{ (M_1,\dots, M_\lambda) \in \Omega_N : M_i = C_k \implies M_{i-1} \neq S_{k+1} \text{ for all } i \in [\lambda], k \geq 2 \right\}.
\end{align*}
For a game $\mathcal R = (M_1, M_2, \dots, M_\lambda) \in \mathcal R_N^{\mathcal P}$ with number of moves $\lambda$, construct the subset of indices $\mathcal I_N(\mathcal R)$ to denote all combine moves in $\mathcal R$:
\begin{align*}
    \mathcal I_N(\mathcal R) = \left\{i \in [\lambda] : M_i = C_k \text{ for some } k \geq 2 \right\}.
\end{align*}
Now, construct the \textit{formal} sequence of moves $\mathcal M(\mathcal R) = (\tilde{M}_1, \tilde{M}_2, \dots, \tilde{M}_\lambda)$ by replacing $M_i$ by a symbol $\mathcal E_k^\ell$ for all $i \in \mathcal I_N(\mathcal R)$: the subscript is the corresponding $k \geq 2$, while $\ell \geq 0$ denotes the longest $(\ell,k)$-prefix that $M_i = C_k$ can be expanded into; note in particular that $\mathcal M(\mathcal R)$ contains no combine moves. Call $\mathcal M(\mathcal R)$ the \textbf{base sequence} of $\mathcal R \in \mathcal R_N$, and let $\mathcal A_N(\mathcal R) \in \mathcal F_N$ denote the collection of all Zeckendorf games resulting from replacing each instance of $\mathcal E_k^\ell$ in $\mathcal M$ by an $(l,k)$-prefix for some $l \leq \ell$. We establish the following analogue of Lemma \ref{partition}, which yields a strictly broader partition of $\Omega_N$ (in the sense that each set in the partition given by Lemma \ref{prefix_partition} is a union of sets in the partition given by Lemma \ref{partition}).

\begin{proposition} \label{prefix_partition}
The sets $\{\mathcal A_N(\mathcal R) : \mathcal R \in \mathcal R_N^{\mathcal P}\}$ partition $\Omega_N$.
\end{proposition}

\begin{proof}
Observe that a base sequence $\mathcal M(\mathcal R) = (\tilde{M}_1, \tilde{M}_2, \dots, \tilde{M}_\lambda)$ uniquely determines $\mathcal R \in \mathcal R_N^{\mathcal P}$ by its explicit split moves moves and the subscripts of each of its symbols $\mathcal E_k^\ell$. This is true up to any initial subsequence, in the sense that there exists at most one $\mathcal R \in \mathcal R_N^{\mathcal P}$ which agrees with the moves and symbols $\mathcal E_k^\ell$ up to subscript.

We show that the sets $\mathcal A_N(\mathcal R)$ are disjoint. Take distinct $\mathcal R_1 = (M_1^1,\dots, M_\lambda^1), \mathcal R_2 = (M_1^2,\dots, M_{\lambda'}^2) \in \mathcal R_N^{\mathcal P}$ with base sequences $\mathcal M_1(\mathcal R_1) = (\tilde{M}_1^1, \tilde{M}_2^1, \dots, \tilde{M}_\lambda^1)$ and $\mathcal M_2(\mathcal R_2) = (\tilde{M}_1^2, \tilde{M}_2^2, \dots, \tilde{M}_{\lambda'}^2)$, respectively. Since $\mathcal R_1 \neq \mathcal R_2$, take smallest $i$ for which $M_i^1 \neq M_i^2$, and study the construction of any two games $\mathcal G_1 \in \mathcal A_N(\mathcal R_1)$ and $\mathcal G_2 \in \mathcal A_N(\mathcal R_2)$ as described above. If $\mathcal G_1$ and $\mathcal G_2$ are consistent prior to $\tilde{M}_i^1$ and $\tilde{M}_i^2$ (i.e., for $k \geq 2$, all $\mathcal E_k^\ell$ are replaced by the same prefix of $C_k$), which is the only way the two games remain equal up to this point, then we produce a difference in the two games on $\tilde{M}_i^1$ and $\tilde{M}_i^2$ if $M_i^1 = \tilde{M}_i^1$ and $M_i^2 = \tilde{M}_i^2$ ($M_i^1 \neq M_i^2$) or $M_i^1 \neq \tilde{M}_i^1$ and $M_i^2 \neq \tilde{M}_i^2$ ($\tilde{M}_i^1 = \mathcal E_{k_1}^{\ell_1}$, $\tilde{M}_i^2 = \mathcal E_{k_2}^{\ell_2}$, $k_1 \neq k_2$). If $M_i^1 = \tilde{M}_i^1$ and $M_i^2 \neq \tilde{M}_i^2$ (say $\tilde{M}_i^2 = \mathcal E_k^\ell$, so $M_i^2 = C_k$), then $\mathcal G_2$ is next filled with either $C_k \neq M_i^1$ ($M_i^1 = M_i^2$ if $C_k = M_i^1$) or an $(l,k)$-prefix for some $l \leq \ell$. In the latter case, if the two games were equal until move $i+l-1$, necessarily $\mathcal R_1$ contains a prefix of length at least $1$, contradicting $\mathcal R_1 \in \mathcal R_N^{\mathcal P}$. 

We therefore conclude that $\mathcal A_N(\mathcal R_1) \cap \mathcal A_N(\mathcal R_2) = \emptyset$, i.e., the sets $\mathcal A_N(\mathcal R)$ for $\mathcal R \in \mathcal R_N$ are disjoint; it remains to show that any game $\mathcal G \in \Omega_N$ is in some set $\mathcal A_N(\mathcal R)$. For $\mathcal G \in \Omega_N$, let $\mathcal R$ be the game resulting from replacing every combine move $C_k$ with the longest playable $(\ell,k)$-prefix. It follows immediately from construction that $\mathcal R \in \mathcal R_N^{\mathcal P}$, and that $\mathcal G \in \mathcal A_N(\mathcal R)$.
\end{proof}

Similarly, we define the generalization for arbitrary $(\ell,k)$-suffixes. Define $\mathcal R_N^{\mathcal S} \in \mathcal F_N$ to be the collection of all Zeckendorf games on input $N$ such that any combine move $C_k$, for $k \geq 2$, cannot be compressed by a $(1,k)$-suffix (and thus by any suffix of $C_k$). Explicitly,
\begin{align*}
    \mathcal R_N^{\mathcal S} = \left\{ (M_1,\dots, M_\lambda) \in \Omega_N : M_i = C_k \implies M_{i+1} \neq S_{k+1} \text{ for all } i \in [\lambda], k \geq 2 \right\}.
\end{align*}
For a game $\mathcal R = (M_1, M_2, \dots, M_\lambda) \in \mathcal R_N^{\mathcal S}$ with number of moves $\lambda$, construct the subset of indices $\mathcal I_N(\mathcal R)$ to denote all combine moves in $\mathcal R$:
\begin{align*}
    \mathcal I_N(\mathcal R) = \left\{i \in [\lambda] : M_i = C_k \text{ for some } k \geq 2 \right\}.
\end{align*}
Now, construct the \textit{formal} sequence of moves $\mathcal M(\mathcal R) = (\tilde{M}_1, \tilde{M}_2, \dots, \tilde{M}_\lambda)$ by replacing $M_i$ by a symbol $\mathcal E_k^\ell$ for all $i \in \mathcal I_N(\mathcal R)$: the subscript is the corresponding $k \geq 2$, while $\ell$ denotes the longest $(\ell,k)$-suffix that $M_i = C_k$ can be expanded into. Call $\mathcal M(\mathcal R)$ the \textbf{base sequence} of $\mathcal R \in \mathcal R_N$, and let $\mathcal A_N(\mathcal R) \in \mathcal F_N$ denote the collection of all Zeckendorf games resulting from replacing each instance of $\mathcal E_k^\ell$ in $\mathcal M$ by an $(l,k)$-suffix for some $l \leq \ell$. We establish the following analogue of Lemma \ref{partition}.

\begin{proposition} \label{suffix_partition}
The sets $\{\mathcal A_N(\mathcal R) : \mathcal R \in \mathcal R_N^{\mathcal S}\}$ partition $\Omega_N$.
\end{proposition}

\noindent We do not provide the proof of Proposition \ref{suffix_partition}, as it is pursued analogously to the proof of Proposition \ref{prefix_partition}. The easy observation that each set $\mathcal A_N(\mathcal R)$, for $\mathcal R \in \mathcal R_N^{\mathcal P}$ or $\mathcal R \in \mathcal R_N^{\mathcal S}$, is a union of equivalence classes of the corresponding sets studied in the proof of Theorem \ref{half_half}, yields the following trivial extension of Lemma \ref{many_decisions}, where we promote the notation $\mathcal A_N(\mathcal G) = \mathcal A_N(\mathcal R)$ for the unique $\mathcal R$ satisfying $\mathcal G \in \mathcal A_N(\mathcal R)$.

\begin{proposition} \label{many_decisions_generalized}
Say $F_n \leq N < F_{n+1}$. For any $c \in (0, \varphi)$ (with $\varphi$ the golden ratio),
\begin{align*}
    & \lim_{N \to \infty} \mu_N \left( \log_2 \left| \mathcal A_N(\mathcal G) \right| \geq c^n \right) = 1,
    & \lim_{N \to \infty} \Prob_N \left( \log_2 \left| \mathcal A_N(\mathcal G) \right| \geq c^n \right) = 1,
\end{align*}
which holds for either of the understandings $\mathcal R \in \mathcal R_N^{\mathcal P}$ and $\mathcal R \in \mathcal R_N^{\mathcal S}$.
\end{proposition}

Finally, we have the following extension of Corollary \ref{num_bins}. We also omit the proof of this result, as it is a straightforward generalization of the proof of the aforementioned result.
\begin{proposition} \label{num_bins_generalized}
Say $F_n \leq N < F_{n+1}$, fix some $\mathcal R \in \mathcal R_N^{\mathcal P}$, and define random variables $X_1^{\mathcal R}, X_2^{\mathcal R}, \dots, X_{N-1}^{\mathcal R}$ corresponding to the length of the expansion corresponding to each combine move for a game in $\mathcal A_N(\mathcal R)$. Under the measures $\mu_N$ and $\Prob_N$ conditioned on $\mathcal A_N(\mathcal R)$, the random variables $X_i^{\mathcal R}$ are independent.
\end{proposition}

\section{Weak Convergence of Random Game Lengths as $N \to \infty$} \label{mixture_gaussians}

In pursuit of the resolution of Conjecture \ref{gaussianity}, it may be productive to ask whether there exist natural subsets of $\Omega_N$ on which the distribution of random game lengths converge weakly to a Gaussian in the limit $N \to \infty$ of infinite input. As discussed in Section \ref{sec:extended_partitions}, the partitions defined in Propositions \ref{prefix_partition} and \ref{suffix_partition} enjoy this property in the sense described in Theorem \ref{weak_mixtures}, restated below.

\weakMixtures*

To explicitly define the distributions referenced in Theorem \ref{weak_mixtures}, define a probability distribution $\Prob_N^{\mathcal R}: 2^{\mathcal A_N(\mathcal R)} \to [0,1]$ via, for any $\mathcal S \in 2^{\mathcal A_N(\mathcal R)}$,
\begin{align*}
    \Prob_N^{\mathcal R}(\mathcal S) = \frac{\Prob_N\left(\mathcal S \cap \mathcal A_N(\mathcal R) \right)}{\Prob_N\left(\mathcal A_N(\mathcal R) \right)}.
\end{align*}
Define the random variable $L_N^{\mathcal R}: \mathcal A_N(\mathcal R) \to \R$ on the space $\left(\mathcal A_N(\mathcal R), 2^{\mathcal A_N(\mathcal R)}, \Prob_N^{\mathcal R}\right)$ by 
\begin{align*}
    L_N^{\mathcal R}\left((M_1,\dots, M_\lambda)\right) = \lambda
\end{align*}
for any game $(M_1,\dots, M_\lambda) \in \mathcal A_N(\mathcal R)$, i.e., $L_N^{\mathcal R}$ studies game lengths in $\mathcal A_N(\mathcal R)$. Then 
\begin{align}
    F_{\mathcal R}(x) = \Prob_N^{\mathcal R}\left( \frac{L_N^{\mathcal R} - \E\left[L_N^{\mathcal R}\right]}{\sqrt{\Var\left( L_N^{\mathcal R} \right)}} \leq x \right).
\end{align}
In other words, when we restrict $\Prob_N$ to the sets $\mathcal A_N(\mathcal R)$ in the natural sense, Theorem \ref{weak_mixtures} states that the distribution of random game lengths enjoys weak convergence to a Gaussian with high probability. We write
\begin{align} \label{eq:cond_dist}
    \mathcal L_N^{\mathcal R} = \frac{L_N^{\mathcal R} - \E\left[L_N^{\mathcal R}\right]}{\sqrt{\Var\left( L_N^{\mathcal R} \right)}} 
\end{align}
to be the random variable $L_N^{\mathcal R}$ normalized to have mean $0$ and variance $1$. Throughout this section, we proceed on the measure $\Prob_N$ and the set $\mathcal R_N^{\mathcal P}$, but the analysis carries over to the uniform measure $\mu_N$ and the suffix partition $\mathcal R_N^{\mathcal S}$.

\begin{proof}[Proof of Theorem \ref{weak_mixtures}]
By Proposition \ref{many_decisions_generalized}, $\lim_{N \to \infty} \Prob_N \left( \log_2 \left| \mathcal A_N(\mathcal G) \right| \geq 1.6^n \right) = 1$, so fix $\mathcal R \in \mathcal R_N$ such that $|\mathcal A_N(\mathcal R)| > 2^{1.6^n}$, and consider the random variable $\mathcal L_N^{\mathcal R}: \mathcal A_N(\mathcal R) \to \R$. Define random variables\footnote{Define these to be strictly positive: if the $i\textsuperscript{th}$ delimiter is replaced with a combine move, say $X_i^{\mathcal R} = 1$.} $X_1^{\mathcal R}, \dots, X_{N-1}^{\mathcal R}$ to correspond to each of the combine moves in $\mathcal R$, with $X_i^{\mathcal R}$ denoting the length of the corresponding prefix for a game in $\mathcal A_N(\mathcal R)$: by Proposition \ref{num_bins_generalized}, the random variables $X_i^{\mathcal R}$ are independent. Expanding Equation (\ref{eq:cond_dist}),
\begin{align} \label{eq:cond_dist_2}
    \mathcal L_N^{\mathcal R} = \frac{L_N^{\mathcal R} - \E\left[L_N^{\mathcal R}\right]}{\sqrt{\Var\left( L_N^{\mathcal R} \right)}} = \frac{ \sum_{i=1}^{N-1} \left( X_i^{\mathcal R} - \E\left[X_i^{\mathcal R} \right] \right)}{\sqrt{\Var\left( L_N^{\mathcal R} \right)}}
\end{align}
where the independence (conditioned on $\mathcal A_N(\mathcal R)$) of the random variables $X_i^{\mathcal R}$ yields that
\begin{align*}
    \Var\left( L_N^{\mathcal R} \right) = \sum_{i=1}^{N-1}  \Var\left( X_i^{\mathcal R} - \E\left[X_i^{\mathcal R} \right] \right)
\end{align*}
and furthermore, taking the maximum over those summands $X_i^{\mathcal R}$ which are non-constant\footnote{Indeed, we understand $\mathcal L_N^{\mathcal R}$ as the sum over the nonconstant random variables amongst those included by the sum in Equation (\ref{eq:cond_dist_2})},
\begin{align} \label{eq:berry_esseen_bd}
    \max_{1 \leq i \leq D} \left( \frac{\E\left[\left|X_i^{\mathcal R} - \E\left[X_i^{\mathcal R} \right]\right|^3\right]}{\E\left[\left(X_i^{\mathcal R} - \E\left[X_i^{\mathcal R} \right]\right)^2\right] \cdot \Var\left( L_N^{\mathcal R} \right)} \right) \leq \max_{1 \leq i \leq D} \frac{2n^3}{\Var\left( L_N^{\mathcal R} \right)} \leq \frac{n^3}{\Var\left( 2^{1.6^n} \right)} 
\end{align}
and the final expression in Equation (\ref{eq:berry_esseen_bd}) certainly vanishes as $N \to \infty$. Thus, the desired result follows immediately from the case of the Berry-Esseen theorem for independent non-identically distributed summands, which namely yields
\begin{align*}
    \sup_{x \in \R} \left|F_{\mathcal R}(x) - \Phi(x) \right| \leq C \cdot \frac{n^3}{\Var\left( 2^{1.6^n} \right)} \xrightarrow{N \to \infty} 0
\end{align*}
for some universal constant $C > 0$.
\end{proof}
Theorem \ref{weak_mixtures} yields that, when restricted to particular natural subsets of games, we have Gaussianity. The scope of this result is admittedly restricted, especially given the result of Corollary \ref{cor:likeliest_rep}. 

\section{Open Problems}

We conclude the work with several potential directions for further inquiry.

\subsection{Other Two-Player Games Based on Recurrences} Theorem \ref{half_half} can be interpreted as saying that if two players proceed mindlessly, the Zeckendorf game is fair in the limit of infinite input; Theorem \ref{many_players} gives that the analogous statement would be true if we were to extend to a $Z$-player Zeckendorf game. Many papers (such as \cite{baily_generalized_2021, GeneralizedZeckendorfGame, boldyriew2020extending}) have extended the paradigm of the two-player Zeckendorf game to other recurrences: we might ask which of these also enjoy this property. In particular, we pose the following conjecture, concerning the two-player Bergman game (\cite{baily_generalized_2021}).

\begin{conjecture} \label{bergman_conj}
In the limit $N \to \infty$ of infinite input, the probabilities of Player 1 and Player 2 winning, under both analogous definitions of random Bergman games, is $1/2$.
\end{conjecture}

The core challenges of proving Conjecture \ref{bergman_conj} are as follows. The principal difference between the two-player Bergman game and the two-player Zeckendorf game is that the move $C_1$ is now a split move which consumes no tokens. Thus, there are structural differences between the two games which affects the range of achievable game lengths. Lemma \ref{many_decisions} also depends on the number of a specific move sequence which may not have an analogue in the Bergman game.

\subsection{Towards Gaussianity: Other Approaches} 

Theorem \ref{weak_mixtures} establishes that certain natural partitions of $\Omega_N$ are such that the components become arbitrarily close to being Gaussian (in the sense of vanishing Kolmogorov-Smirnov distance) with arbitrarily high probability: it is unclear how to extend this to the entirety of $\Omega_N$. Proposition \ref{suffix} remarks that all games in $\Omega_N$ can be achieved by permutations of suffixes from a shortest game, which suggests the following question.

\begin{question}
Can we extend the techniques in Section \ref{mixture_gaussians} to a partition founded on the greedy embedding of Proposition \ref{suffix}? In particular, could studying the moments of the corresponding components lead to a proof of Conjecture \ref{gaussianity}?
\end{question}

Another possible direction is to restrict our attention to certain subsets of moves across all of $\Omega_N$, rather than certain subsets of $\Omega_N$. In particular, a sufficiently strong affirmative answer to Question \ref{gaussian_moves} would resolve Conjecture \ref{gaussianity}.

\begin{question} \label{gaussian_moves}
Is the distribution of the number of occurrences of a particular combine or split move asymptotically Gaussian, under either the measure $\mu_N$ or $\Prob_N$?
\end{question}

\subsection{Towards Gaussianity: Mixing}

We outline one more possible approach towards proving Conjecture \ref{gaussianity}, which relies on the literature surrounding \textit{mixing central limit theorems}: these are analogues of the central limit theorem concerning sums of dependent random variables, applicable if the dependencies amongst the summands are sufficiently well-behaved. In particular, the main result of \cite{ekstrom2014general} states the following.
\begin{theorem}[\cite{ekstrom2014general}]
Let $\{X_{N,i}: 1 \leq i \leq d_N\}$ be a triangular array of random variables defined on the probability space $(\Omega, \mathcal F, \Prob)$, $\Bar{X}_{N,d_N} = \frac{1}{d_N} \sum_{i=1}^{d_N} X_{N,i}$, and $\alpha_N: \N \to \R$ by 
\begin{align} \label{ekstrom:alpha}
    \alpha_N(k) = \sup_m \sup_{A \in \mathcal F_0^m(N), B \in \mathcal F_{m+k}^\infty(N)} \left|\Prob(A \cap B) - \Prob(A) \Prob(B)\right|
\end{align}
where $\mathcal F_{m_1}^{m_2}(N) = \sigma\left(X_{N,m_1}, X_{N,m_1+1}, \dots, X_{N,m_2} \right)$ denotes the $\sigma$-algebra generated by random variables $X_{N,m_1}, X_{N,m_1+1}, \dots, X_{N,m_2}$. If there exists constants $C_1, C_2 > 0$, $\delta > 0$ such that
\begin{align} \label{ekstrom:cond1}
    \E\left[\left| X_{N,i} - \E[X_{N,i}] \right|^{2+\delta}\right] < C_1
\end{align}
and 
\begin{align} \label{ekstrom:cond2}
    \sum_{k=0}^\infty (k+1)^2 \alpha_N^{\frac{\delta}{4+\delta}}(k) < C_2,
\end{align}
then 
\begin{align} \label{ekstrom:result}
    \sqrt{d_N}\left(\Bar{X}_{N,d_N} - \E\Bar{X}_{N,d_N} \right) \xrightarrow[N \to \infty]{\mathcal D} N\left(0, \Var\left( \sqrt{d_N} \Bar{X}_{N,d_N} \right) \right).
\end{align}
\end{theorem}
We can define a triangular array of Bernoulli random variables by letting $X_{N,i} = 1$ if and only if the $i$th move of a Zeckendorf game with input $N$ is a splitting move (in particular, $X_{N,i}$ maps to $0$ on the event that the Zeckendorf game on input $N$ terminates prior to move $i$), with $d_N$ (the number of random variables in the $N$th row) the length of the longest game on input $N$; Conjecture \ref{gaussianity} is thus equivalent to establishing that the row sums $\sum_{i=1}^{d_N} X_{N,i}$ weakly converge to a Gaussian, since the number of combine moves for a given input $N$ is constant. This is precisely the statement of Equation \ref{ekstrom:result}. Also, the events $A$ and $B$ in Equation \ref{ekstrom:alpha} can be understood as fixed $\{0,1\}$-realizations of subsets of random variables in $\{X_{N,1}, \dots, X_{N,m}\}$ and $\{X_{N,m+k}, \dots, X_{N,d_N}\}$, respectively. 

Since Equation \ref{ekstrom:cond1} certainly holds under this setup (the constituent random variables are Bernoulli), resolving Conjecture \ref{gaussianity} can be reduced to the following statement.
\begin{question} \label{mixing_approach}
For the triangular array $\{X_{N,i}: 1 \leq i \leq d_N\}$ defined above, does there exist a constant $\delta > 0$ and a constant $C > 0$ such that 
\begin{align} \label{gauss_eq}
    \sum_{k=0}^\infty (k+1)^2 \alpha_N^{\frac{\delta}{4+\delta}}(k) < C
\end{align}
for all natural numbers $N \in \N$?
\end{question}
By the preceding discussion, answering Question \ref{mixing_approach} in the affirmative by establishing \ref{gauss_eq} immediately yields Conjecture \ref{gaussianity}. In particular, it would suffice to establish the statement of Question \ref{mixing_reduction} in the affirmative, although it is not immediately clear if this should be true.

\begin{question} \label{mixing_reduction}
Does there exist a constant $\epsilon > 0$ and a constant $C > 0$ such that
\begin{align}
    \left|\Prob_N(A \cap B) - \Prob_N(A) \Prob_N(B)\right| \leq \frac{C}{k^{2+\epsilon}}
\end{align}
whenever events $A$ and $B$ correspond to fixed $\{0,1\}$-realizations of subsets of the Bernoulli random variables in $\{X_{N,1}, \dots, X_{N,m}\}$ and $\{X_{N,m+k}, \dots, X_{N,d_N}\}$, respectively?
\end{question}

Of course, the techniques introduced in the main body of this paper are quite elementary, and there may be other promising approaches not included in this section that would contribute towards resolving Conjecture \ref{gaussianity}.

\section*{Acknowledgements}

This research was conducted as part of the 2022 SMALL REU program at Williams College, and was funded by NSF Grant DMS1947438, Harvey Mudd College, and Williams College funds. We thank Professor Steven J. Miller and our colleagues from the 2022 SMALL REU program, and in particular those pursuing the Zeck Shifts project, for many helpful insights.

\section*{References}

\printbibliography[heading=none]

\end{document}